\documentclass[10pt]{article}

\usepackage[a4paper,margin=1in]{geometry}
\usepackage{amsmath,amssymb,amsthm,mathtools}
\usepackage{bm}
\usepackage{enumitem}
\usepackage{microtype}
\usepackage[hidelinks]{hyperref}
\usepackage[nameinlink,capitalise,noabbrev]{cleveref}
\usepackage{authblk}
\usepackage{geometry}
\usepackage{enumitem} 
\usepackage{hyperref, cleveref}
\hypersetup{colorlinks=true,linkcolor=blue,citecolor=blue,urlcolor=blue}

\numberwithin{equation}{section}

\newtheorem{theorem}{Theorem}[section]
\newtheorem{proposition}[theorem]{Proposition}
\newtheorem{lemma}[theorem]{Lemma}
\newtheorem{corollary}[theorem]{Corollary}
\newtheorem{remark}[theorem]{Remark}

\newtheorem{hypothesis}{Hypothesis}

\newcommand{\R}{\mathbb R}
\newcommand{\C}{\mathbb C}
\newcommand{\HH}{\mathcal H}
\newcommand{\OO}{\mathcal O}
\newcommand{\YY}{\mathcal Y}
\newcommand{\cA}{\mathcal A}
\newcommand{\Dom}{\operatorname{Dom}}
\newcommand{\dist}{\operatorname{dist}}
\newcommand{\norm}[1]{\left\lVert #1\right\rVert}

\newenvironment{keywords}{\par\noindent\textbf{Keywords:}\ }{\par\vspace{1em}}

\theoremstyle{plain}
\newtheorem{problem}{Problem}
\newtheorem{definition}{Definition}[section]

\title{Inverse Heat Source Problems from Boundary Flux and Interior Observations on Sets of Low Hausdorff Dimension}

\author[1]{Ze Li\thanks{rikudosennin@163.com}}
\author[2]{Zhiyuan Li\thanks{lizhiyuan@nbu.edu.cn}}
\author[3]{Zhi Zhou\thanks{Corresponding author, zhizhou@polyu.edu.hk}}
\affil[1,2]{School of Mathematics and Statistics, Ningbo University, Ningbo 315211, China}
\affil[3]{ Department of Applied Mathematics, The Hong Kong Polytechnic University}
\date{ }

\begin{document}
\maketitle

\begin{abstract}
This paper investigates conditional stability for inverse source problems for the heat equation with a known temporal factor and an unknown spatial  component in a bounded $C^{1,1}$ domain. We focus on observations supported on sets of low Hausdorff dimension and establish conditional stability in this setting.
For boundary observations on compact sets of positive $q$-dimensional Hausdorff content, we establish logarithmic stability from full-time boundary flux observations and double-logarithmic stability from delayed-time boundary flux observations. The admissible dimensional ranges are $q>d-2$ when the observation set is contained in a flat boundary patch and $q>d-1-c_{d+1}$ on a general $C^{1,1}$ boundary, where $c_{d+1}>0$ depends only on the dimension. A key ingredient in deriving these results is a new boundary spectral inequality for the Dirichlet Laplacian, which controls a finite Dirichlet spectral sum through observations of the normal derivative of its elliptic extension on such a boundary set. Our results also cover inverse heat source problems with interior observations on sets of positive $q$-dimensional Hausdorff content for some $q>d-1$, yielding logarithmic stability from full-time observations for general sources in $H_0^1(\Omega)$ and H\"older stability from terminal-time observations for sources in a suitable spectral Gevrey class.
\end{abstract}

\begin{keywords}
heat equation; inverse source problem; Hausdorff content; boundary Hausdorff spectral inequality; conditional stability
\end{keywords}

\section{Introduction}\label{sec:intro}

Let  $T>0$ and let $\Omega\subset\mathbb{R}^d$ be a bounded domain with sufficiently smooth boundary $\partial\Omega$. We consider the initial-boundary value problem for the heat equation with a separated source
\begin{equation}\label{qforward}
\begin{cases}
\partial_t u - \Delta u = g(t)f(x), & (x,t)\in \Omega\times(0,T),\\
u(x,0)=0, & x\in\Omega,\\
u(x,t)=0, & (x,t)\in \partial\Omega\times (0,T),
\end{cases}
\end{equation}
where the function $u(x,t)$ describes the spatial and temporal distribution of some substance such as contaminants. The time-dependent factor $g(\cdot)$ is known as a $C^1[0,T]$ function with $g(0)\neq0$, while the spatial source distribution $f$ is unknown.
The forward problem \eqref{qforward} models heat conduction in a medium with an internal heat source whose temporal profile is understood but whose spatial localisation is inaccessible to direct measurement~\cite{alfanov1994inverse,woodbury2023inverse}.

Letting $\Gamma\subset\partial \Omega$ of
Hausdorff dimension greater than $d-2$, we are concerned with the following inverse source problems  (ISPs) by taking observations on the spatial subset $\Gamma$.

\begin{problem}[ISPs with Hausdorff boundary observations]\label{prob2main}
Given the observation set $\Gamma\subset\partial \Omega$ of
Hausdorff dimension greater than $d-2$ and the temporal factor $g(t)$, determine the unknown source term $f(x)$ in \eqref{qforward} from one of the following types of observation data:
\begin{enumerate}[label=ISP\arabic*.,itemsep=0pt,leftmargin=4em]
    \item Full-time data: $ \partial_\nu u(x,t)$ for $(x,t)\in\Gamma\times(0,T)$;
    \item Delayed-time data: $ \partial_\nu u(x,t)$ for $(x,t)\in\Gamma\times(\tau,T)$ with $\tau>0$.
\end{enumerate}
\end{problem}

Letting $\omega\subset\Omega$ be a measurable set with Hausdorff dimension $s\in(d-1,d)$ and positive $s$-dimensional Hausdorff measure, we are concerned with the following inverse source problems (ISPs) by taking observations on the spatial subset $\omega$.
\begin{problem}[ISPs with Hausdorff interior observations]\label{prob:main}
Given the observation set $\omega\subset\Omega$ of positive Hausdorff measure and the temporal factor $g(t)$, determine the unknown source term $f(x)$ in \eqref{qforward} from one of the following types of observation data:
\begin{enumerate}[label=ISP\arabic*.,start=3,itemsep=0pt,leftmargin=4em]
    \item Full-time data:  $ u(x,t)$ for $(x,t)\in\omega\times(0,T)$;
    \item Terminal-time data: $u(x,t)$ for $(x,t)\in\omega\times\{T\}$.
\end{enumerate}
\end{problem}

These inverse problems arise naturally in applications such as the detection of nuclear leaks or the identification of heat-generating devices through walls; see, e.g., \cite{alfanov1994inverse,woodbury2023inverse}.
However, inverse problems are often severely ill-posed: without additional \textsl{a priori} assumptions on the unknwon parameter, arbitrarily small perturbations in the measurements may lead to large reconstruction errors  \cite{isakov2006inverse}. Conditional stability can nevertheless be recovered by restricting the unknown source to a suitable a priori bounded admissible set. More precisely, such estimates control the reconstruction error by a modulus of continuity of the measurement error, often in a weaker topology. Beyond their theoretical significance, conditional stability estimates provide an essential tool for analyzing convergence rates of regularization methods \cite{ChengYamamoto:2000, WangChenChengJiang:2025, JinKeretaXia:2026} and deriving error bounds for finite-dimensional numerical reconstructions \cite{Burman:2023, Burman:2025, JinZhou:2021, JinLuQuanZhou:2026,Cen2026,CenZhou:2024}.

For the parabolic inverse source problems, stability estimates with observations taken on subdomains or open sub-boundaries have been extensively investigated, where most existing works rely heavily on Carleman estimates.
Imanuvilov and Yamamoto~\cite{imanuvilov1998lipschitz} proved global Lipschitz stability for general source terms $f\in H_0^1(\Omega)$ by Carleman estimate with regular weight function, where measurements are imposed on open subboundary. Yamamoto and Zou \cite{yamamoto2001simultaneous} further derived stability results for parabolic inverse problems via singular Carleman weight functions. For a systematic overview of Carleman estimates for parabolic inverse problems, we refer readers to the comprehensive survey \cite{yamamoto2009carleman} and the classic monograph \cite{klibanov2004carleman}.
Recently, Huang et al.~\cite{huang2020stability} refined the classical argument and obtained local H\"older from partial boundary measurements; their cut-off-free approach greatly simplifies the proof structure for stability analysis.
In addition to Carleman-based arguments, stability analysis for parabolic inverse source problems with partial data can also be carried out by Fourier series methods. Yamamoto~\cite{yamamoto1993conditional} constructed conditional logarithmic stability through solution decomposition and rigorous Fourier coefficient estimates, which establishes a non-Carleman framework for ill-posed inverse problems.

When only fixed-time data are available, existing stability results typically require observations over the whole domain $\Omega$; see, for example, \cite[Chapter 9]{isakov2006inverse}. Cheng and Liu \cite{cheng2020inverse} showed that local observations on a subdomain at a single fixed time are in general insufficient, and established uniqueness and conditional stability using local measurements at two distinct times. For the more general source term $R(x,t)f(x)$, stability results have also been obtained using Carleman estimates with $x$-independent weight functions \cite{yamamoto2009carleman}, as well as logarithmic convexity and related methods \cite{carasso1999logarithmic,chen2022simultaneous,levine1970logarithmic,phung2018carleman}. In this more general setting, however, a global fixed-time observation is typically combined with additional lateral boundary data collected over a time interval of positive length.


All aforementioned theoretical stability results assume that the observation
region contains an open set in $\Omega$ or $\partial\Omega$.
This assumption is essential for the applicability of Carleman estimates:
the construction of the weight function $\phi = e^{\lambda\psi}$ requires a function $d(x)$ satisfying $|\nabla d|>0$ on a neighbourhood of the observation region~\cite[Lemma~4.1]{yamamoto2009carleman}, and the repeated integration-by-parts in the derivation of the Carleman inequality~\cite[Section~3]{yamamoto2009carleman} demands sufficient regularity of the observation set to control the boundary integrals.
Both requirements break down when $\omega$ is merely of positive Hausdorff measure; in particular, $\omega$ may be totally disconnected with empty interior.
The same obstruction affects the spectral method based on Fourier decomposition and unique continuation~\cite{cheng2020inverse,isakov2006inverse}, which also relies on open-set observations.

Recently, the works \cite{BurqMoyano2023, green2024observability,LogunovMalinnikova2018} established observability inequalities for the heat equation from interior sets of positive Hausdorff measure. These interior-type spectral inequalities are adopted in our analysis for interior observation settings (Theorems \ref{thmgene} and \ref{ganal}). For boundary observations, however, no analogous spectral inequality
appears to be available when the observation set is merely of positive Hausdorff content and may have zero surface measure and empty relative interior. The principal novelty of this paper is the establishment of a \emph{boundary Hausdorff spectral inequality (BHSI)} for the Dirichlet Laplacian on $C^{1,1}$ domains. This inequality controls finite Dirichlet spectral sums through normal-derivative observations on thin boundary sets and is derived using quantitative propagation-of-smallness results from
\cite{LogunovMalinnikova2018}. To the best of our knowledge, it is the first boundary spectral inequality of this type for observation sets of positive Hausdorff content. As applications of the newly established BHSI, we derive conditional stability estimates for the full-time and delayed-time boundary inverse source problems in Problem~\ref{prob2main}. In parallel, using the interior spectral inequality from \cite{green2024observability}, we establish conditional stability for the interior observation in Problem~\ref{prob:main}.

Our main results are summarized as follows.  Let $A$ denote the Dirichlet realization in $L^2(\Omega)$ of the elliptic operator appearing in the equation. In the stability estimates below, $\varepsilon$ denotes the norm of the observation specified in the
corresponding item.

\begin{enumerate}[label=ISP\arabic*.,itemsep=0pt,leftmargin=4em]
    \item \emph{Full-time boundary observation
    (Theorem~\ref{insourth}):}
For $f\in\operatorname{Dom}(A^\sigma)$ with $\sigma>0$, under an appropriate trace condition, we prove a logarithmic stability estimate of the form
    $$
    \|f\|_{L^2(\Omega)}
    \leq C|\log\varepsilon|^{-\sigma},
    $$
    using full-time measurements of the time derivative of the boundary
    flux $\partial_t\partial_\nu u$ on $\Gamma\times(0,T)$.

    \item \emph{Delayed-time boundary observation
    (Theorem~\ref{deHaus}):}
    For $f\in\operatorname{Dom}(A^\sigma)$ with $\sigma>d/4$, we prove a double-logarithmic stability estimate of the form
    $$
    \|f\|_{L^2(\Omega)}
    \leq C\bigl|\log|\log\varepsilon|\bigr|^{-\sigma},
    $$
    using delayed-time measurements of the time derivative of the
    boundary flux $\partial_t\partial_\nu u$ on
    $\Gamma\times(\tau,T)$.

    \item \emph{General Sobolev source and full-time interior observation
    (Theorem~\ref{thmgene}):}
    For $f\in H_0^1(\Omega)$, we prove a logarithmic stability estimate of the form
    $$
    \|f\|_{L^2(\Omega)}
    \leq C|\log\varepsilon|^{-1/2},
    $$
    using full-time measurements of the time derivative of the solution
    $\partial_tu$ on $\omega\times(0,T)$.

    \item \emph{Gevrey source and terminal-time interior observation
    (Theorem~\ref{ganal}):}
For $f\in G^{1,\rho}(\Omega)$ with $\rho$ sufficiently large, we prove a H\"older stability estimate of the form
    $$
    \|f\|_{L^2(\Omega)}
    \leq C\varepsilon^\alpha
    $$
    for some $\alpha>0$, using only the terminal observation
    $u(\cdot,T)$ on $\omega$.
\end{enumerate}

\subsection{Notations and function spaces}
Let $\Omega\subset\mathbb{R}^d$ be a bounded smooth domain and $T>0$.
We denote the space-time cylinder by $\Omega_T = \Omega\times(0,T)$.
The standard $L^2(\Omega)$ inner product and norm are $\langle\cdot,\cdot\rangle$ and $\|\cdot\|_{L^2(\Omega)}$ respectively.
The Sobolev space $H_0^1(\Omega)$ is equipped with the norm $\|u\|_{H_0^1(\Omega)} = \|\nabla u\|_{L^2(\Omega)}$, and $H^2(\Omega)$ denotes the usual second-order Sobolev space, see details in Adams \cite{adams1975sobolev} for example.

Let $A:H_0^1(\Omega)\cap H^2(\Omega) \to L^2(\Omega)$ be the positive Dirichlet Laplacian on $\Omega$.
It admits a sequence of eigenvalues $\{\lambda_k\}_{k=1}^\infty$ with
$0<\lambda_1\le\lambda_2\le\cdots\to\infty$,
and a corresponding orthonormal basis $\{\phi_k\}_{k=1}^\infty\subset H_0^1(\Omega)$ of eigenfunctions:
$$
A\phi_k = \lambda_k\phi_k \text{ in } \Omega, \quad \phi_k|_{\partial\Omega}=0,
$$
see e.g., Evans \cite{evans2010partial}.
For $f\in L^2(\Omega)$, let $f_k = \langle f,\phi_k\rangle$ denote its Fourier coefficients.

We recall the definitions of Hausdorff content and Hausdorff dimension.
Let $E\subset\mathbb R^d$ and $s\geq 0$. The $s$-dimensional Hausdorff
content of $E$ is defined by
$$
\mathcal{C}^s_{\mathcal{H}} (E)
:=
\inf\left\{
\sum_{j=1}^{\infty} r_j^s
:
E\subset \bigcup_{j=1}^{\infty}B(x_j,r_j)
\right\}.
$$
The Hausdorff dimension of $E$ is defined by
$$
\dim_{\mathcal H}(E)
:=
\inf\{s\ge0:{\mathcal C}^s_{\mathcal{H}}(E)=0\}.
$$

\subsection{ Main results on boundary observation}

In this section, we state our main stability results for boundary observation, which rely on the following geometric condition on $\Gamma$.

\begin{hypothesis}[Geometric setting]\label{hyg}
Let $d\ge 2$. We assume that $\Gamma\subset\partial\Omega$ is compact satisfying either
$$
\mathcal C_{\mathcal H}^q(\Gamma)>0,
\qquad q>d-2,
$$
when $\Gamma$ lies in a flat boundary patch, or
$$
\mathcal C_{\mathcal H}^q(\Gamma)>0,
\qquad q>d-1-c_{d+1},
$$
on a general $C^{1,1}$ boundary, where $c_{d+1}\in (0,1)$ depends only on $d$.
\end{hypothesis}

Fix $S_0>0$.  For $\Lambda\ge1$ and a finite family
$a=(a_j)_{\lambda_j\le\Lambda}\subset\C$, define $\mathcal A_{\Lambda}(a)$  by
\begin{equation}
\cA_\Lambda(a)
:=
\left(\sum_{\lambda_j\le\Lambda}|a_j|^2\right)^{1/2}.
\label{A}
\end{equation}
For a compact set $\Gamma\subset\partial\Omega$, set
\begin{equation}
\OO_{\Gamma,S_0}(\Lambda,a)
:=
\sup_{\substack{x\in\Gamma\\ |s|\le S_0}}
\left|\partial_\nu U_{\Lambda,a}(x,s)\right|,
\label{qO-Lambda}
\end{equation}
where $U_{\Lambda,a}$ is defined by
\begin{equation*}
U_{\Lambda,a}(x,s)
:=
\sum_{\lambda_j\le\Lambda}
 a_j\phi_j(x)\cosh(\sqrt{\lambda_j}\,s).
\end{equation*}
For the finite spectral sums considered here the normal derivative has a continuous representative on every $C^{1,1}$ boundary patch.  One way to see this is to bootstrap the eigenvalue equation in $W^{2,p}$ locally up to the boundary, with $p>d$; see also the boundary regularity estimates used in \cite{ChenLiuLiu2026}.

Our first main result is a boundary Hausdorff spectral inequality (BHSI),  which serves as a key ingredient in establishing the conditional stability  estimates for the inverse source problems with boundary observations.
\begin{theorem} [BHSI]
\label{cbhsi}
Let $\Omega\subset\R^d$ be a  $C^{1,1}$ bounded domain and $\Gamma\subset\partial\Omega$  satisfy {\bf Hypothesis}  \ref{hyg}.  Then, for every $S_0>0$, there is $C\ge1$ such that
\begin{equation*}
\cA_\Lambda(a)
\le
C e^{C\sqrt\Lambda}\OO_{\Gamma,S_0}(\Lambda,a)
\end{equation*}
for all $\Lambda\ge1$ and all finite coefficient families $(a_j)_{\lambda_j\le\Lambda}\subset\C$.
\end{theorem}

Now we consider the case that our boundary observation is taken on the full time $(0,T)$.
Let $g\in W^{1,\infty}(0,T)$ with
\begin{equation}
g(0)\ne0,
\label{g0nonze}
\end{equation}
and consider
\begin{equation}
\begin{cases}
\partial_tu+Au=g(t)f,&0<t<T,\\
u(0)=0.
\end{cases}
\label{2source}
\end{equation}
The Dirichlet boundary condition is incorporated in $A$.

The second theorem gives conditional stability for the inverse source problem with observations on $\Gamma \times (0,T)$.
The natural boundary quantity is the time derivative of the measured flux.
For complete functional clarity, define
\begin{equation}
X_\Gamma:=C(\Gamma;\mathbb C)
\label{XGama}
\end{equation}
and let
\begin{equation}
\mathfrak D_{\Gamma,T}
:=
\left\{
f\in L^2(\Omega):
 t\mapsto\partial_\nu e^{-tA}f|_\Gamma
 \in L^2(0,T;X_\Gamma)
\right\}.
\label{ASC}
\end{equation}

\begin{theorem}\label{insourth}
Let $\Omega\subset\R^d$ be a  $C^{1,1}$ bounded domain and $\Gamma\subset\partial\Omega$  satisfy {\bf Hypothesis}  \ref{hyg}. Let $T>0$, $\sigma>0$, and $M>0$. Fix $g\in W^{1,\infty}(0,T)$ with \eqref{g0nonze}.  Assume
\begin{equation*}
f\in\Dom(A^\sigma)\cap\mathfrak D_{\Gamma,T},
\qquad
\norm{A^\sigma f}_2\le M.
\end{equation*}
Let $u$ solve \eqref{2source}, and define
\begin{equation*}
\delta_{\mathrm{src}}
:=
\norm{\partial_t\partial_\nu u}_{L^2(0,T;C(\Gamma))}.
\end{equation*}
Then
\begin{equation}
\norm{f}_{L^2(\Omega)}
\le
CM
\left[
\log\left(e+\frac{M}{\delta_{\mathrm{src}}}\right)
\right]^{-\sigma}.
\label{stabi}
\end{equation}
The constant depends only on the fixed geometry, the quantitative constants
in BHSI of \Cref{cbhsi}, the time $T$, the exponent $\sigma$, and $g$.
\end{theorem}

The following theorem gives conditional stability for the inverse source problem with observations on $\Gamma \times (\tau,T)$ with $0<\tau <T$.
The extra assumption needed in the delayed-time problem concerns the time
factor $g$.  For $\theta\in(0,\pi/2)$and $R>T$, set
$$
\mathcal S_{\theta,R}
:=
\left\{
z\in\mathbb C:
0<|z|<R,\ |\arg z|<\theta
\right\}.
$$
We assume that $g$ extends holomorphically to
$\mathcal S_{\theta,R}$, continuously to $0$, and that
\begin{equation}\label{Sg}
g(0)\neq0,
\qquad
\sup_{z\in\mathcal S_{\theta,R}}
\bigl(|g(z)|+|g'(z)|\bigr)
\le G
\end{equation}
for some constant $G>0$.

\begin{theorem}
\label{deHaus}
Let $\Omega\subset \mathbb R^d$ be a $C^{1,1}$ bounded domain and $\Gamma\subset \partial \Omega$ fulfill the {\bf Hypothesis}  \ref{hyg}. Let $g$ satisfy  \eqref{Sg}.   Assume that
$$
\sigma>\frac d4,
\qquad
f\in\operatorname{Dom}(A^\sigma),
\qquad
\|A^\sigma f\|_{L^2(\Omega)}\le M.
$$
Let $u$ solve \eqref{2source}, and define
\begin{equation*}
\delta_\tau
:=
\left(
\int_\tau^T
\sup_{x\in\Gamma}
\left|
\partial_t\partial_\nu u(x,t)
\right|^2\,dt
\right)^{1/2}.
\end{equation*}
Then there exist constants $C,D>0$, depending only on
$$
\Omega,\Gamma,d,\sigma,\tau,T,\theta,R,G,
$$
and on the constant in BHSI of \Cref{cbhsi}, such that
\begin{equation}
\|f\|_{L^2(\Omega)}
\le
CM
\left[
\log\left(
e+
\log\left(
e+\frac{DM}{\delta_\tau}
\right)
\right)
\right]^{-\sigma}.
\label{dedoub}
\end{equation}
The right-hand side is interpreted as $0$ when $\delta_\tau=0$.
Particularly, the delayed boundary data determine $f$ uniquely. Moreover, as $\delta_\tau/M\downarrow0$,
$$
\|f\|_{L^2(\Omega)}
\le
CM
\left[
\log\log\left(\frac{DM}{\delta_\tau}\right)
\right]^{-\sigma}.
$$
\end{theorem}

\begin{remark}[Why the condition $\sigma>d/4$ appears]
The restriction
$
\sigma>\frac d4
$
is not a Hausdorff-dimension restriction.  It is imposed solely by the use
of the observation space
$
L^2(0,T;X_\Gamma))
$
after analytic continuation all the way to $t=0$.  The boundary smoothing
estimate gives
$$
\|\partial_\nu e^{-tA}f\|_{C(\Gamma)}
\lesssim
t^{-\beta}\|A^\sigma f\|_2
$$
with some $\beta<1/2$ precisely when
$\sigma>d/4$.  The inequality $\beta<1/2$ is what guarantees square
integrability at $t=0$.

If one replaces the norm $L^2_t$ by $L^p_t$ with a smaller
$p$, this regularity threshold can be lowered correspondingly.  Such a
variant is independent of the BHSI argument and is not pursued here.
\end{remark}

\begin{remark}[Origin of the second logarithm]

The two logarithms in \eqref{dedoub} have different
origins.  The first one is generated by quantitative analytic continuation
of the delayed boundary measurement from $(\tau,T)$ toward $t=0$.
This is the mechanism introduced in the delayed boundary inverse-source
argument of Choulli--Yamamoto \cite{ChoulliYamamoto2004}.  The second
logarithm comes from the intrinsic backward instability of the heat
semigroup, through the low--high spectral splitting of
Lemma~\ref{lemlog}.  The Hausdorff geometry enters through
BHSI and hence through
Proposition~\ref{potiin}. It changes the admissible
boundary sets and the constants, but not the final exponent $\sigma$ of
the outer logarithm.
\end{remark}

\subsection{Main results on interior observation}
We introduce the spectral analytic class associated with the Dirichlet Laplacian on $\Omega$.

\begin{definition}[Spectral Gevrey class $G^{1,\rho}$]\label{gevrey}
We define spectral  Gevrey class $G^{1,\rho}(\Omega)$ with $\rho>0$ by
$$
G^{1,\rho}(\Omega)
:=
\left\{
f\in L^2(\Omega):
\sum_{k=1}^\infty
e^{2\rho\sqrt{\lambda_k}}
|f_k|^2<\infty
\right\},\quad
\text{with}~~
f_k=\langle f,\phi_k\rangle.
$$
It is equipped with the norm
$$
\|f\|_{G^{1,\rho}(\Omega)}
:=
\left(
\sum_{k=1}^\infty
e^{2\rho\sqrt{\lambda_k}}
|f_k|^2
\right)^{1/2}.$$
\end{definition}

The following spectral inequality from recent observability results for the heat equation in Green et al. \cite{green2024observability} is fundamental for studying analytic sources. It connects the $L^\infty$-norm of a band-limited function to its supremum on a set of positive Hausdorff measure.
\begin{proposition}[\cite{green2024observability}] \label{specpro}
Let $\gamma>0$ and $\Omega\subset \mathbb R^d$. If the set $\omega\subset\Omega$ is such that $\mathcal{C}^{d-1+\gamma}_{\mathcal{H}}(\omega)>0$, then there  exists a constant $C_{S}>0$, depending only on $\Omega,\gamma$ and $\omega$, such that for every $\Lambda>0$ and every $v\in L^2(\Omega)$,
$$
\|\Pi_\Lambda v\|_{L^2(\Omega)} \le C_{S} e^{C_{S}\sqrt{\Lambda}} \sup_{x\in\omega} |\Pi_\Lambda v(x)|,
$$
where $\Pi_\Lambda$ denotes the orthogonal projection onto $\operatorname{span}\{\phi_k : \lambda_k\le\Lambda\}$.
\end{proposition}

We study two inverse source problems for the heat equation
\begin{equation}\label{mainpde}
\begin{cases}
\partial_t u + A u = g(t)f(x), & (x,t)\in\Omega_T,\\
u(x,0)=0, & x\in\Omega,\\
u(x,t)=0, & (x,t)\in \partial\Omega\times(0,T),
\end{cases}
\end{equation}
where $g\in C^1[0,T]$ with $g(0)\neq0$ is known, and $f$ is the unknown spatial source.

\medskip\noindent
\textbf{Case 1: General source.}
The observation consists of $u$ on the whole time interval and on the Hausdorff-thick set $\omega$:
$$
\text{Data: } u(x,t), \quad (x,t)\in\omega\times(0,T).
$$

\begin{theorem}[Full-time observation, general $f$]\label{thmgene}
Let $\gamma>0$ and $\Omega\subset \mathbb R^d$ be a given domain.  Let $\omega\subset \Omega$ with $\mathcal{C}^{d-1+\gamma}_{\mathcal{H}}(\omega)>0$. Assume that  $f\in H_0^1(\Omega)$ with $\|f\|_{H_0^1(\Omega)}\le M$ for some $M>0$, and  $g\in C^1[0,T]$ with $g(0)\neq0$.
Then there exists a constant $C,D>0$ depending on $\Omega,\omega,\gamma,T,g,M$ such that
$$
\|f\|_{L^2(\Omega)} \le C \left[\log \frac{D}{\displaystyle\int_0^T \sup_{x\in\omega} \bigl|\partial_t u(x,t)\bigr|\,dt}\right]^{-\frac12}.
$$
\end{theorem}

\medskip\noindent
\textbf{Case 2: Analytic source.}
If the source is analytic, we can substantially relax the observation, with only the terminal state measured on $\omega$:
$$
\text{Data: } u(x,T), \quad x\in\omega.
$$

\begin{theorem}[Terminal observation, analytic $f$]\label{ganal}
Let $\gamma>0$ and $\Omega\subset \mathbb R^d$ be a given domain.  Let $\omega\subset \Omega$ with $\mathcal{C}^{d-1+\gamma}_{\mathcal{H}}(\omega)>0$. Assume that $f\in G^{1,\rho}(\Omega)$ with $\rho>C_{S}$ ($C_{S}$ being the constant from Proposition~\ref{specpro} ) and assume $\|f\|_{G^{1,\rho}}\le M$.
Suppose that $g\in L^1(0,T)$ and there exists $c_g>0$ such that
\begin{equation}\label{3cog}
\left|\int_0^T g(s)e^{-\lambda_k(T-s)}ds\right| \ge \frac{c_g}{\lambda_k} ,\quad\forall k\in\mathbb N.
\end{equation}
Then there exist constants $C>0 $  depending on $C_S, \rho, c_g, \|g\|_{L^1(0,T)}$ and $\alpha \in (0,1)$ depending on $C_S,\rho$, such that
$$
\|f\|_{L^2(\Omega)} \le CM^{1-\alpha}\left(\sup_{x\in\omega}|u(x,T)|\right)^{\alpha}.
$$
\end{theorem}
\begin{remark}\label{comp}
The contrast between the two theorems reveals the role of the a priori regularity:
\begin{itemize}
    \item General $f$ requires full-time observation of $u$ and the stability is of order $|\log(\text{error})|^{-1/2}$.
    \item Analytic $f$ requires only the terminal observation $u(T)$ on $\omega$ and enjoys the sharper order $|{error}|^{\alpha}$ with $\alpha>0$.
\end{itemize}
Both results are new in allowing the observation set $\omega$ to be merely of positive Hausdorff measure instead of an open set.

\end{remark}

We point out that the hypothesis \eqref{3cog}  in Theorem~\ref{ganal}
are readily satisfied by a wide class of functions $g$.
For instance, if $g\in C^1[0,T]$ with $g(T)\neq0$ and $g$ does not change sign
on $[0,T]$, then integration by parts gives
\begin{align*}
b_k(T):&=\int_0^T g(s)\,e^{-\lambda_k(T-s)}ds\\
&= \frac{g(T)-g(0)e^{-\lambda_k T}}{\lambda_k} - \frac{1}{\lambda_k}\int_0^T g'(s)\,e^{-\lambda_k(T-s)}ds
= \frac{g(T)}{\lambda_k} + O\!\left(\frac{1}{\lambda_k^2}\right)
\end{align*}
as $\lambda_k\to\infty$.  Therefore, for all sufficiently large $k$,
$$
|b_k(T)| \ge \frac{|g(T)|}{2\lambda_k}.
$$
The sign condition on $g$ guarantees that $b_k(T)\neq0$ for the finitely many
remaining indices.  Thus, after possibly reducing the constant $c_g$, we obtain
$|b_k(T)|\ge \frac{c_g}{\lambda_k}$ for every $k$, and hypothesis \eqref{3cog}  is
verified.

\subsection{Organization}
The remainder of this paper is organized as follows.
Section~\ref{secbhsi} is devoted to the proof of the boundary
Hausdorff spectral inequality stated as Theorem~1.1.  After collecting the necessary spectral notation and
Hausdorff content tools, the proof is split according to the boundary
geometry: Section~\ref{subsec-flat} establishes the flat-patch case
(Theorem~\ref{thm:flat-bhsi}), while Section~\ref{subsec-c11} treats general $C^{1,1}$ boundaries (Theorem~\ref{thm:c11-bhsi}).  These two theorems together yield Theorem~1.1.
Section~\ref{secinte} proves some parabolic interpolation results by applying BHSI.
Section~\ref{secinte} applies this estimate to the inverse source problem with full-time boundary observations and proves Theorem 1.2.
Section~\ref{delsour} treats delayed-time boundary observations and proves Theorem~\ref{deHaus}.  Section~\ref{proof} establishes the stability results for interior observations, namely Theorem~\ref{thmgene} and  Theorem~\ref{ganal}.  Finally, Section~\ref{sinverinitial} presents a further application to the inverse initial data problem and establishes logarithmic conditional stability from Hausdorff boundary flux observations.
\label{inverinitial}

\section{Boundary Hausdorff spectral inequalities}
\label{secbhsi}
In this section we prove the boundary Hausdorff spectral inequality
announced as Theorem~1.1 in the introduction.  We first introduce the
spectral quantities and Hausdorff content used throughout the proof.
The main analytic input is a pair of quantitative propagation results
for gradients of elliptic equations from thin sets, recalled in
Section~\ref{subsec-prelim}.  The proof is then divided into two geometric
cases: the case where the boundary contains a flat patch is handled in
Section~\ref{subsec-flat}, and the general $C^{1,1}$ boundary case is
treated in Section~\ref{subsec-c11} by local flattening and an
odd-reflection argument.

\subsection{Notation and basic estimates}\label{subsec-prelim}

Fix $S_0>0$.  For $\Lambda\ge1$ and a finite family
$a=(a_j)_{\lambda_j\le\Lambda}\subset\C$, define $\mathcal A_{\Lambda}(a)$  by \eqref{A}, and
define the auxiliary elliptic extension
\begin{equation*}
U_{\Lambda,a}(x,s)
:=
\sum_{\lambda_j\le\Lambda}
a_j\phi_j(x)\cosh(\sqrt{\lambda_j}\,s).
\end{equation*}
Then
\begin{equation*}
(\Delta_x+\partial_s^2)U_{\Lambda,a}=0
\quad\text{in }\Omega\times\R,
\qquad
U_{\Lambda,a}=0
\quad\text{on }\partial\Omega\times\R.
\end{equation*}

We shall use two quantitative propagation results for gradients of
elliptic equations from thin sets. The flat and curved cases differ only
in the propagation theorem used after reflection.

\begin{proposition}
\label{grad}
Let $d\ge 2$.
\begin{enumerate}[label=\textup{(\roman*)}]
\item
Let $h$ be harmonic in $B_2\subset\R^d$.  Suppose that
$E\subset B_{1/2}$ is contained in an affine hyperplane and that, for some
$\eta>0$ and $m_0>0$,
$$
\mathcal C^{d-2+\eta}_\HH(E)\ge m_0.
$$
Then there exist $C\ge1$ and $\theta\in(0,1)$, depending only on
$d,\eta,m_0$, such that
\begin{equation*}
\sup_{B_{1/2}}|\nabla h|
\le
C\bigl(\sup_E|\nabla h|\bigr)^\theta
 \bigl(\sup_{B_2}|\nabla h|\bigr)^{1-\theta}.
\end{equation*}

\item
There exists $c_d\in(0,1)$, depending only on $d$, with the following property.  Let
$$
\operatorname{div}(\mathbf A\nabla h)=0
\quad\text{in }B_2,
$$
where $\mathbf A$ is symmetric and, uniformly elliptic, with fixed
ellipticity and $W^{1,\infty}$ bounds.  If
$$
E\subset B_{1/2},
\qquad
\mathcal C^{d-1-c_d+\eta}_\HH(E)\ge m_0
$$
for some $\eta>0$, then
\begin{equation*}
\sup_{B_{1/2}}|\nabla h|
\le
C\bigl(\sup_E|\nabla h|\bigr)^\theta
 \bigl(\sup_{B_2}|\nabla h|\bigr)^{1-\theta},
\end{equation*}
where $C$ and $\theta\in(0,1)$ depend only on the fixed geometric and ellipticity parameters, $d,\eta, m_0$.
\end{enumerate}
\end{proposition}
\begin{proof}
Part \textup{(i)} follows by Lemma 2.3 of \cite{green2024observability}, which is a corollary of  Malinnikova's generalized
Cauchy--Riemann propagation theorem  \cite{Malinnikova2004}.    Part \textup{(ii)} is the gradient
propagation theorem of Logunov--Malinnikova \cite{LogunovMalinnikova2018}, in
its fixed geometry Hausdorff-content form.
\end{proof}

We also use two standard auxiliary facts.
\begin{lemma}
\label{procon}
Let $F\subset\R^m$ be compact and let $J\subset\R$ be a bounded
interval.  If
$$
\mathcal C^q_\HH(F)\ge m_0>0,
$$
then
\begin{equation}
\mathcal C^{q+1}_\HH(F\times J)
\ge c\,m_0\,|J|,
\label{contpro}
\end{equation}
where $c>0$ depends only on $m$ and $q$.
\end{lemma}

\begin{proof}
By Frostman's lemma, see e.g., \cite{falconer1990fractal} and  \cite{mattila1995geometry}, there is a finite positive measure $\mu$ supported on $F$ such that
$$
\mu(F)\ge c_0m_0,
\qquad
\mu(B(x,r))\le r^q.
$$
For $\nu=\mu\otimes \mathcal L^1|_J$ one has
$$
\nu(B((x,t),r))\le 2r^{q+1},
\qquad
\nu(F\times J)\ge c_0m_0|J|.
$$
The mass-distribution principle \cite{falconer1990fractal} gives \eqref{contpro}.
\end{proof}

The following lemma gives the low frequency upper bound  and the interior concentration lower bound of $U_{\Lambda,a}$.

\begin{lemma}
\label{lowfre}
Let  $S_1>0$. Then the following statements hold.
\begin{enumerate}[label=\textup{(\roman*)}]
\item
If $K$ is a fixed compact subset of a coordinate cylinder contained in
$\overline\Omega\times(-S_1,S_1)$ and separated from the artificial end faces
$s=\pm S_1$, then
\begin{equation}
\sup_K\bigl(|U_{\Lambda,a}|+|\nabla_{x,s}U_{\Lambda,a}|\bigr)
\le
C_K e^{C_K\sqrt\Lambda}\cA_\Lambda(a),
\label{liftgrow}
\end{equation}
where $\cA_\Lambda(a)$ is defined by \eqref{A}.
\item
For every nonempty open set $\omega\Subset\Omega$,
\begin{equation}
\cA_\Lambda(a)
\le
C_\omega e^{C_\omega\sqrt\Lambda}
\norm{U_{\Lambda,a}(\cdot,0)}_{L^2(\omega)}.
\label{intespec}
\end{equation}
\end{enumerate}
\end{lemma}

\begin{proof}
 \eqref{liftgrow} is proved in Lemma 2.8 of \cite{ChenLiuLiu2026}.  Estimate
\eqref{intespec} is the classical Lebeau--Robbiano spectral
inequality. In fact, \cite{ApraizEscauriazaWangZhang2014} proved the Lebeau--Robbiano spectral
inequality for bounded Lipschitz locally star-shaped domains, which  particularly applies to bounded
$C^{1,1}$ domains.
\end{proof}

\subsection{Flat boundary patches}
\label{subsec-flat}

We first treat the simplified setting where $\partial\Omega$ contains a
relatively open flat patch.  In this case, after odd reflection, the
harmonic extension remains harmonic in a full neighbourhood, allowing a
direct application of the gradient propagation theorem for harmonic
functions.
\begin{theorem}[Flat-patch BHSI]
\label{thm:flat-bhsi}
Assume that $\Omega\subset\R^d$ is bounded and connected and that
$\partial\Omega$ contains a relatively open flat patch $\Sigma$.  Let
$\Gamma\Subset\Sigma$ be compact.  Suppose that, for some $q>d-2$,
\begin{equation*}
\mathcal C ^q_\HH(\Gamma)>0.
\end{equation*}
Then, for every $S_0>0$, there is $C\ge1$ such that
\begin{equation}
\cA_\Lambda(a)
\le
C e^{C\sqrt\Lambda}\OO_{\Gamma,S_0}(\Lambda,a)
\label{qflat-bhsi}
\end{equation}
for every $\Lambda\ge1$ and every finite coefficient family
$(a_j)_{\lambda_j\le\Lambda}\subset\C$, where $\OO_{\Gamma,S_0}$ is defined in \eqref{qO-Lambda}.
\end{theorem}
\begin{proof}
It suffices to consider real coefficients $\{a_j\}$. In fact, the complex case follows by
applying the real estimate to real and imaginary parts.

After a rigid motion there is $r_0>0$ such that, in a neighborhood of the
observation set,
$$
\Omega=\{(x',x_d):x_d>0\},
\qquad
\partial\Omega=\{x_d=0\},
$$
and $\Gamma\subset\{x_d=0\}$.  Choose
$0<s_*<\min\{S_0,r_0/8\}$ and put $J_*=[-s_*,s_*]$.  In the lifted ambient
dimension
$$
\tilde d = d+1
$$
the set
$$
E:=\Gamma\times J_*
$$
is contained in the hyperplane $\{x_d=0\}\subset\R^{\tilde d}$.  By
\Cref{procon},
$$
\mathcal C^{q+1}_{\HH} (E)>0,
\qquad
q+1=\tilde d-2+\eta,
\qquad
\eta=q-(d-2)>0.
$$

Let $V(x',x_d,s)=U_{\Lambda,a}(x',x_d,s)$ for $x_d>0$ and define its odd
reflection by
$$
\widetilde V(x',x_d,s)
=
\begin{cases}
V(x',x_d,s),&x_d\ge0,\\
-V(x',-x_d,s),&x_d<0.
\end{cases}
$$
Because $V=0$ on $x_d=0$, the reflected function is harmonic in a full
neighborhood of $E$.  Moreover, on $x_d=0$ all tangential derivatives,
including the $s$ derivative, vanish.  Hence
\begin{equation*}
|\nabla_{x,s}\widetilde V(x',0,s)|
=
|\partial_\nu U_{\Lambda,a}(x',0,s)|.
\end{equation*}

After a fixed rescaling, by \Cref{grad} (i) there exist
compact sets $K\Subset K_1$ and constants $C\ge1$, $\theta\in(0,1)$ such
that
\begin{equation*}
\sup_K|\nabla\widetilde V|
\le
C\OO_{\Gamma,S_0}(\Lambda,a)^\theta
\bigl(\sup_{K_1}|\nabla\widetilde V|\bigr)^{1-\theta}.
\end{equation*}
By \Cref{lowfre} (i), $K_1$ can be further chosen to satisfy
$$
\sup_{K_1}|\nabla\widetilde V|
\le
Ce^{C\sqrt\Lambda}\cA_\Lambda(a).
$$
Choose a fixed nonempty ball $\omega_*\Subset\Omega$ in the upper part of
$K$, sufficiently close to the flat boundary.

Noting that
$$U_{\Lambda,a}(x',x_d,0)=\int^{x_d}_0 \partial_{r}U_{\Lambda,a}(x',r,0)dr,
$$
we have
$$
\norm{U_{\Lambda,a}(\cdot,0)}_{L^2(\omega_*)}
\le
C\sup_K|\nabla\widetilde V|.
$$
Applying
\eqref{intespec} and combining the preceding estimates, we then obtain
\begin{align}
\cA_\Lambda(a)
&\le C_{\omega_*} e^{C_{\omega_*}\sqrt\Lambda}
\norm{U_{\Lambda,a}(\cdot,0)}_{L^2(\omega_*)}\le C e^{C \sqrt\Lambda}\sup_K|\nabla\widetilde V|\le C e^{C \sqrt\Lambda} \OO_{\Gamma,S_0}(\Lambda,a)^\theta
\bigl(\sup_{K_1}|\nabla\widetilde V|\bigr)^{1-\theta}\nonumber\\
&\le
Ce^{C\sqrt\Lambda}
\OO_{\Gamma,S_0}(\Lambda,a)^\theta
\cA_\Lambda(a)^{1-\theta}.\label{dssdd}
\end{align}
Then
\eqref{qflat-bhsi} follows by \eqref{dssdd}.
\end{proof}

\subsection{General \texorpdfstring{$C^{1,1}$}{C11} boundaries}\label{subsec-c11}

For general $C^{1,1}$ boundaries, the flattening and reflection
procedure leads to a divergence-form elliptic equation with Lipschitz
coefficients.  We therefore appeal to the second part of
Proposition~\ref{grad}, which is designed for such equations.

\begin{theorem}[$C^{1,1}$ BHSI]
\label{thm:c11-bhsi}
Let $\Omega\subset\R^d$ be bounded, connected and of class $C^{1,1}$.  Set
$\tilde d=d+1$ and let $c_{\tilde d}$ be the dimensional constant from
\Cref{grad} (ii).  Let $\Gamma\subset\partial\Omega$ be
compact and suppose that, for some
\begin{equation*}
q>d-1-c_{\tilde d},
\end{equation*}
one has $\mathcal C^q_{\HH}(\Gamma)>0$.  Then, for every $S_0>0$, there is
$C\ge1$ such that
\begin{equation}
\cA_\Lambda(a)
\le
C e^{C\sqrt\Lambda}\OO_{\Gamma,S_0}(\Lambda,a)
\label{c11bhsi}
\end{equation}
for all $\Lambda\ge1$ and all finite coefficient families $(a_j)_{\lambda_j\le\Lambda}\subset\C$.
\end{theorem}

\begin{proof}
Choose a finite $C^{1,1}$ boundary atlas.  Since Hausdorff content is
subadditive, one chart contains a compact subset $\Gamma_0\subset\Gamma$ with
$\mathcal C^q_\HH (\Gamma_0)>0$.  The boundary flattening map $\Phi$  is bi-Lipschitz, so
its image $\widehat\Gamma_0\subset\R^{d-1}$ also has positive $q$-content.
Choose $J_*=[-s_*,s_*]\Subset(-S_0,S_0)$.  By
\Cref{procon},
\begin{equation}
\mathcal C^{q+1}_{\HH}(\widehat\Gamma_0\times J_*)>0.
\label{lift}
\end{equation}
Since $\tilde d=d+1$,
\begin{equation}
q+1=\tilde d-1-c_{\tilde d}+\eta,
\qquad
\eta:=q-(d-1-c_{\tilde d})>0.
\label{eta}
\end{equation}
Denote $\tilde{U}(x,s)=U(\Phi^{-1}(x),s)$.
In normal boundary coordinates, extended trivially in $s$, the function
$\tilde U_{\Lambda,a}$ solves a divergence-form equation
$$
\operatorname{div}(\mathbf A\nabla v)=0
$$
in the upper half-cylinder, where $\mathbf A$ is symmetric, uniformly
elliptic and Lipschitz.  By the Dirichlet odd-reflection lemma of
Chen--Liu--Liu \cite[Lemma~2.3]{ChenLiuLiu2026}, the odd reflection
$\widetilde v$ solves
$$
\operatorname{div}(\widetilde{\mathbf A}\nabla\widetilde v)=0
$$
in a full ball, with $\widetilde{\mathbf A}\in W^{1,\infty}$ uniformly
elliptic, and on the reflecting hyperplane
\begin{equation*}
|\nabla\widetilde v|
\sim
|\partial_\nu U_{\Lambda,a}|.
\end{equation*}
The constants depend only on the fixed $C^{1,1}$ chart.

Now \eqref{lift}--\eqref{eta} allow us to apply
\Cref{grad} (ii).  Exactly as in the flat case, we obtain
fixed compact sets $K\Subset K_1$ and $\theta\in(0,1)$ such that
$$
\sup_K|\nabla\widetilde v|
\le
C\OO_{\Gamma,S_0}(\Lambda,a)^\theta
\bigl(\sup_{K_1}|\nabla\widetilde v|\bigr)^{1-\theta}.
$$
The estimate \eqref{liftgrow} controls the large factor by
$Ce^{C\sqrt\Lambda}\cA_\Lambda(a)$.  Integrating along short normal curves
from the Dirichlet boundary to a fixed interior ball $\omega_*\Subset\Omega$
and applying \eqref{intespec}, we conclude
$$
\cA_\Lambda(a)
\le
Ce^{C\sqrt\Lambda}
\OO_{\Gamma,S_0}(\Lambda,a)^\theta
\cA_\Lambda(a)^{1-\theta}.
$$
Then  \eqref{c11bhsi} follows.
\end{proof}

The following remark gives the literature origin of the two dimension thresholds of flat patches and general $C^{1,1}$ domains.
\begin{remark}
The threshold $d-2$ in \Cref{thm:flat-bhsi} comes from Malinnikova's
codimension-two hyperplane propagation theorem after adding the auxiliary
variable.  The weaker $C^{1,1}$ threshold $d-1-c_{\tilde d}$ comes solely from the
currently available Logunov--Malinnikova theorem for gradients of solutions
with Lipschitz coefficients.  Chen--Liu--Liu \cite{ChenLiuLiu2026} provides the $C^{1,1}$ boundary
flattening and odd-reflection mechanism used above.
\end{remark}

\section{From BHSI to parabolic interpolation}
\label{secinte}

From now on, $\Gamma$ is any compact boundary set for which either
\eqref{qflat-bhsi} or \eqref{c11bhsi} holds.  We write this common
estimate as
\begin{equation}
\cA_\Lambda(a)
\le
C_Be^{C_B\sqrt\Lambda}
\sup_{\substack{x\in\Gamma\\ |s|\le S_0}}
\left|
\sum_{\lambda_j\le\Lambda}
 a_j\cosh(\sqrt{\lambda_j}s)\partial_\nu\phi_j(x)
\right|.
\tag{BHSI}
\label{qBH-SI}
\end{equation}

For a time interval $I\Subset(0,\infty)$ define
\begin{equation*}
\norm{F}_{\YY_\Gamma(I)}
:=
\left(
\int_I \sup_{x\in\Gamma}|F(x,t)|^2\,dt
\right)^{1/2}.
\end{equation*}
If this quantity is infinite, all estimates below are understood in the
extended sense and are trivially true.

\subsection{Analytic continuation of the boundary heat trace}

The parabolic step follows the time analytic mechanism used by
Chen--Liu--Liu \cite{ChenLiuLiu2026}: the auxiliary factor
$\cosh(s\sqrt\lambda)$ is recovered from all time derivatives of the heat
boundary trace.

We first prove the uniform holomorphy of the boundary flux.
\begin{lemma}
\label{boho}
Let $0<a<b<\infty$ and let $I_0\Subset(a,b)$. Assume that
$\partial\Omega$ is of class $C^{1,1}$ in a neighborhood of
$\overline{\Gamma}$.
Then there exist $r>0$ and $C>0$ such that, for every
$h\in L^2(\Omega)$, the map
$$
z\longmapsto
\left.\partial_\nu(e^{-zA}h)\right|_\Gamma
$$
is holomorphic from
$$
\{z\in\mathbb C:\dist(z,I_0)<r\}
$$
into $C(\Gamma)$ and
$$
\sup_{\dist(z,I_0)<r}
\|\partial_\nu(e^{-zA}h)\|_{C(\Gamma)}
\le C\|h\|_{L^2(\Omega)}.
$$
In particular, for every $x\in\Gamma$,
$$
F_x(z):=\partial_\nu(e^{-zA}h)(x)
$$
is holomorphic on the same neighborhood, uniformly in $x$.
\end{lemma}

\begin{proof}
Choose $\delta>0$ and $r>0$ sufficiently small so that
$$
\Re z>2\delta
$$
whenever $\dist(z,I_0)<r$, and so that the shifted neighborhood
$$
\{z-\delta:\dist(z,I_0)<r\}
$$
is compactly contained in the sector of analyticity of the semigroup $e^{-tA}$.

Fix $p>d$. By the positive time $L^2$--$L^p$ smoothing properties of analytic heat
semigroups \cite{O2025} and the $W^{2,p}$ Dirichlet regularity on $C^{1,1}$ domains
\cite{GT}, for every $\delta>0$ and every $p>d$ one has
$$
\|e^{-\delta A}f\|_{W^{2,p}(\Omega\cap U)}
\le C_{\delta,p}\|f\|_{L^2(\Omega)}.
$$
Since $p>d$, the embedding
$$
W^{2,p}(\Omega\cap U)\hookrightarrow C^{1,\alpha}
(\overline{\Omega\cap U'})
$$
holds for every $U'\Subset U$, with
$\alpha=1-d/p>0$. Hence
$$
T_\delta f
:=
\left.\partial_\nu(e^{-\delta A}f)\right|_\Gamma
$$
defines a bounded operator
$$
T_\delta:L^2(\Omega)\longrightarrow C(\Gamma).
$$

For $\dist(z,I_0)<r$, the semigroup property gives
$$
\left.\partial_\nu(e^{-zA}h)\right|_\Gamma
=
T_\delta e^{-(z-\delta)A}h.
$$
The map
$$
z\longmapsto e^{-(z-\delta)A}h
$$
is $L^2(\Omega)$-valued holomorphic on this neighborhood.
Since $T_\delta$ is bounded, the boundary-flux map is
$C(\Gamma)$-valued holomorphic. Moreover, boundedness of the analytic
semigroup on compact subsectors yields
$$
\sup_{\dist(z,I_0)<r}
\|\partial_\nu(e^{-zA}h)\|_{C(\Gamma)}
\le
C\|h\|_{L^2(\Omega)}.
$$
Finally, evaluation at any $x\in\Gamma$ is a bounded linear functional
on $C(\Gamma)$, which proves the asserted scalar holomorphy.
\end{proof}

We also need the one-dimensional analytic propagation estimate.
\begin{lemma}
\label{anpro}
Let $J\subset\R$ be a nondegenerate compact interval and let
$t_*\in\operatorname{int}J$.  Fix a complex neighborhood $\mathcal U$ of
$J$.  There exist $C>0$, $\vartheta\in(0,1)$ and $R>0$, depending only on
$J,t_*,\mathcal U$, such that every holomorphic function $F$ on $\mathcal U$
with $\sup_{\mathcal U}|F|\le M$ satisfies
\begin{equation*}
|F^{(k)}(t_*)|
\le
C k! R^{-k}
M^{1-\vartheta}\norm{F}_{L^2(J)}^{\vartheta},\quad k\ge0.
\end{equation*}
\end{lemma}
\begin{proof}
Choose $\rho>0$ such that
$$
[t_*-4\rho,t_*+4\rho]\subset J,
\qquad
\overline{D(t_*,4\rho)}\subset\mathcal U,
$$
where $D(t_*,r):=\{ z\in \mathbb C: |z-t_*|<r\}$.
By Cauchy's estimates,
$$
|F^{(j)}(t)|
\le C M j!\rho^{-j},
\qquad
|t-t_*|\le 3\rho,\quad j\ge0.
$$
The quantitative propagation of smallness estimate for real-analytic
functions of Vessella \cite{Vessella1999} (see also
\cite[Theorem~4]{ApraizEscauriazaWangZhang2014}), applied in one
dimension to the real and imaginary parts of $F$, yields
$$
\sup_{|t-t_*|\le2\rho}|F(t)|
\le
C M^{1-\alpha}\|F\|_{L^2(J)}^\alpha
$$
for some $\alpha\in(0,1)$ depending only on $J,t_*$ and $\mathcal U$.

Applying the two-constants theorem   for subharmonic functions $\log |F(z)|$
(e.g.   \cite[Theorem~4.3.7]{Ransford1995}) in the  upper half-disk
 $D^+(t_*,2\rho)=\{z: |z-z_*|<2\rho, \Im z>0\}$, and using the uniform lower bound
for the harmonic measure of the diameter on $\overline{D^+(t_*,\rho)}$, one has
$$
\sup_{D^+(t_*,\rho)}|F|
\le
C M^{1-\vartheta}\|F\|_{L^2(J)}^\vartheta
$$
with some $\vartheta\in(0,1)$ depending only on the fixed $t_*,J$ and $\mathcal{U}$.
Then with the same argument in the lower half-disks  $D^-(t_*,2\rho)$ and $\overline{D^-(t_*,\rho)}$, we obtain
$$
\sup_{D(t_*,\rho)}|F|
\le
C M^{1-\vartheta}\|F\|_{L^2(J)}^\vartheta.
$$
Cauchy's derivative estimate therefore gives
$$
|F^{(k)}(t_*)|
\le
C k!\rho^{-k}
M^{1-\vartheta}\|F\|_{L^2(J)}^\vartheta,
\qquad k\ge0.
$$
\end{proof}

 The following proposition gives the positive time interpolation from Hausdorff boundary flux.
\begin{proposition}
\label{potiin}
Let $0<t_1<t_*<t_2$.  Then there exist $C>0$ and $\beta\in(0,1)$ such that,
for every $h\in L^2(\Omega)$,
\begin{equation*}
\begin{aligned}
{
\norm{e^{-t_*A}h}_{L^2(\Omega)}
\le
C\norm{h}_{L^2(\Omega)}^{1-\beta}
\norm{\partial_\nu e^{-tA}h}_{\YY_\Gamma((t_1,t_2))}^{\beta}.
}
\end{aligned}
\end{equation*}
\end{proposition}

\begin{proof}
Set
$$
M:=\norm{h}_{L^2(\Omega)},
\qquad
\delta:=\norm{\partial_\nu e^{-tA}h}_{\YY_\Gamma((t_1,t_2))}.
$$
If $M=0$ there is nothing to prove.  Choose a compact interval
$J\Subset(t_1,t_2)$ containing $t_*$ in its interior.  For each $x\in\Gamma$,
applying \Cref{anpro} to
$$
F_x(t)=\partial_\nu(e^{-tA}h)(x),
$$
and using  \Cref{boho} and the inequality
$\norm{F_x}_{L^2(J)}\le\delta$, one obtains that there exist $C,R>0$ and
$\vartheta\in(0,1)$ such that
\begin{equation*}
\sup_{x\in\Gamma}|F_x^{(k)}(t_*)|
\le
C k! R^{-k}M^{1-\vartheta}\delta^\vartheta,
\qquad k\ge0.
\end{equation*}
Since
$$
F_x^{(k)}(t_*)
=(-1)^k\partial_\nu(A^ke^{-t_*A}h)(x),
$$
the normally convergent series
\begin{equation*}
\partial_\nu\bigl(\cosh(s\sqrt A)e^{-t_*A}h\bigr)(x)
=
\sum_{k=0}^\infty
\frac{(-1)^ks^{2k}}{(2k)!}F_x^{(k)}(t_*)
\end{equation*}
gives, for $|s|\le S_0$,
\begin{equation}
\sup_{\substack{x\in\Gamma\\|s|\le S_0}}
\left|
\partial_\nu\bigl(\cosh(s\sqrt A)e^{-t_*A}h\bigr)(x)
\right|
\le
C M^{1-\vartheta}\delta^\vartheta.
\label{fab}
\end{equation}
Indeed,
$\sum_{k\ge0}k!R^{-k}S_0^{2k}/(2k)!<\infty$.

Let $\Pi_\Lambda$ be the spectral projector of $A$ onto
$\{\lambda_j\le\Lambda\}$.  The high frequency auxiliary tail satisfies
\begin{equation}
\sup_{\substack{x\in\Gamma\\|s|\le S_0}}
\left|
\partial_\nu\cosh(s\sqrt A)(I-\Pi_\Lambda)e^{-t_*A}h(x)
\right|
\le
Ce^{-c\Lambda}M
\label{eqat}
\end{equation}
for all sufficiently large $\Lambda$.
To justify this, choose $p>d$,  the boundary estimate
$\|\partial_\nu w\|_{C(\Gamma)}\le C\|w\|_{W^{2,p}(\Omega_0)}$ on a fixed
$C^{1,1}$ neighborhood $\Omega_0$ of $\Gamma$, see e.g. \cite{GT},  reduces
$\|\partial_\nu w\|_{C(\Gamma)}$ to the $L^p$ norms of $w$ and $Aw$. Then  splitting
$e^{-t_*A}=e^{-t_*A/2}e^{-t_*A/2}$, we have
\begin{align*}
&\|\cosh(s\sqrt{A})(I-\Pi_{\Lambda})e^{-t_*A}h \|_{L^p}+\|A\cosh(s\sqrt{A})(I-\Pi_{\Lambda})e^{-t_*A}h \|_{L^p}\\
&\le C \|\|\cosh(s\sqrt{A})(I-\Pi_{\Lambda})e^{-\frac{1}{2}t_*A}h \|_{L^2}+C\|A\cosh(s\sqrt{A})(I-\Pi_{\Lambda})e^{-\frac{1}{2}t_*A}h \|_{L^2}\\
&\le C  e^{-c \Lambda} \|h\|_{L^2(\Omega)}
\end{align*}
where in the last inequality we have used that  on $\{\lambda>\Lambda\}$ the multipliers
$$
(1+\lambda)e^{-t_*\lambda/2}\cosh(S_0\sqrt\lambda)
$$
are bounded by $Ce^{-c\lambda}$.  This proves \eqref{eqat}.

Applying \Cref{cbhsi} to
$a_j=e^{-t_*\lambda_j}(h,\phi_j)$ for $\lambda_j\le\Lambda$, we obtain  from
\eqref{fab}--\eqref{eqat} that
\begin{equation*}
\norm{P_\Lambda e^{-t_*A}h}_{L^2}
\le
Ce^{C\sqrt\Lambda}
\left(M^{1-\vartheta}\delta^\vartheta+e^{-c\Lambda}M\right).
\end{equation*}
The complementary part obeys
\begin{equation*}
\norm{(I-P_\Lambda)e^{-t_*A}h}_{L^2}
\le e^{-t_*\Lambda}M.
\end{equation*}
Hence
\begin{equation}
\norm{e^{-t_*A}h}_{L^2}
\le
Ce^{C\sqrt\Lambda}M^{1-\vartheta}\delta^\vartheta
+Ce^{-c_0\Lambda+C\sqrt\Lambda}M
\label{preop}
\end{equation}
for some $c_0>0$.

When $0<\delta<M/2$, choose
$$
\Lambda=K\log(M/\delta)
$$
with a fixed sufficiently large $K$.  Since
$\sqrt L\le\varepsilon L+C_\varepsilon$ for $L\ge0$, the first term on the
right of \eqref{preop} is bounded by
$CM^{1-\beta}\delta^\beta$ for any fixed $\beta<\vartheta$, after choosing
$\varepsilon$ small.  The second term has the same bound if $K$ is large
enough.  If $\delta\ge M/2$, the contraction estimate
$\norm{e^{-t_*A}h}_2\le M$ gives the conclusion after increasing $C$.
\end{proof}

\subsection{A backward logarithmic estimate}
\label{sec:backward-log}

\begin{lemma}
\label{lemlog}
Let $\sigma>0$, $t_*>0$, and let $h\in\Dom(A^\sigma)$ satisfy
\begin{equation*}
\norm{A^\sigma h}_{L^2(\Omega)}\le M.
\end{equation*}
Then there exists $C=C(\sigma,t_*,\lambda_1)>0$ such that
\begin{equation}
{}
\norm{h}_{L^2(\Omega)}
\le
CM
\left[
\log\left(e+\frac{M}{\norm{e^{-t_*A}h}_{L^2(\Omega)}}\right)
\right]^{-\sigma}.
\label{backlogg}
\end{equation}
\end{lemma}

\begin{proof}
Set $\eta=\norm{e^{-t_*A}h}_2$.  For every $\Lambda\ge\lambda_1$,
\begin{align*}
\norm{h}_2^2
\le
\sum_{\lambda_j\le\Lambda}|(h,\phi_j)|^2
+
\Lambda^{-2\sigma}\norm{A^\sigma h}_2^2
\le
e^{2t_*\Lambda}\eta^2+M^2\Lambda^{-2\sigma}.
\end{align*}
If $\eta\ge M/2$, \eqref{backlogg} follows from
$\norm{h}_2\le\lambda_1^{-\sigma}M$.  If $0<\eta<M/2$, choose
$$
\Lambda=\max\left\{\lambda_1,
\frac{1}{2t_*}\log\frac{M}{\eta}\right\}.
$$
Then the low-frequency term is bounded by $M\eta$, while
$r\le C_\sigma[\log(1/r)]^{-2\sigma}$ for $0<r\le1/2$ shows that
$M\eta$ is bounded by the same logarithmic modulus as the high-frequency
term.  Then  \eqref{backlogg} follows.  The
trivial case $\eta=0$ follows by injectivity of $e^{-t_*A}$.
\end{proof}

\section{Inverse separated source from Hausdorff boundary flux}
\label{inversour}

Let $g\in W^{1,\infty}(0,T)$ with
\begin{equation*}
g(0)\ne0,
\end{equation*}
and consider
\begin{equation*}
\begin{cases}
\partial_tu+Au=g(t)f,&0<t<T,\\
u(0)=0.
\end{cases}
\end{equation*}
The Dirichlet boundary condition is incorporated in $A$.

The natural boundary quantity is the time derivative of the measured flux.
Recall $X_{\Gamma}$ and $\mathfrak D_{\Gamma,T}$ defined in \eqref{XGama} and \eqref{ASC} respectively.

The following lemma gives the Volterra inversion in an abstract function space setting.
\begin{lemma}
\label{lem:volterra}
Let $X$ be a Banach space and let $q,p\in L^2(0,T;X)$ satisfy
\begin{equation}
p(t)
=
g(0)q(t)+\int_0^t g'(t-s)q(s)\,ds.
\label{vol}
\end{equation}
Then
\begin{equation}
\norm{q}_{L^2(0,T;X)}
\le
C_g\norm{p}_{L^2(0,T;X)},
\label{volinv}
\end{equation}
where $C_g$ depends only on $T$, $|g(0)|^{-1}$ and
$\norm{g'}_{L^\infty(0,T)}$.
\end{lemma}

\begin{proof}
Write \eqref{vol} as
$(I+K)q=p/g(0)$, where
$$
(Kq)(t)=\frac1{g(0)}\int_0^t g'(t-s)q(s)\,ds.
$$
The $n$-fold Volterra kernel has $L^1(0,T)$ norm bounded by
$$
\frac1{n!}
\left(
\frac{T\norm{g'}_\infty}{|g(0)|}
\right)^n.
$$
Young's inequality therefore implies convergence in operator norm of
$\sum_{n\ge0}(-K)^n$ on $L^2(0,T;X)$, which proves
\eqref{volinv}.
\end{proof}

\noindent
{\bf Proof of Theorem \ref{insourth}.}
\begin{proof}
Let $0<t_1<t_*<t_2<T$, and
$
v(t)=e^{-tA}f.
$
Duhamel's formula gives
\begin{equation*}
u(t)=\int_0^t g(t-s)v(s)\,ds.
\end{equation*}
For $f\in\mathfrak D_{\Gamma,T}$ the normal traces belong to
$L^2(0,T;C(\Gamma))$, and differentiation of the Banach-valued convolution
gives
\begin{equation*}
\partial_t\partial_\nu u(t)
=
g(0)\partial_\nu v(t)
+
\int_0^t g'(t-s)\partial_\nu v(s)\,ds
\end{equation*}
in $L^2(0,T;C(\Gamma))$.  Applying \Cref{lem:volterra} with
$X=C(\Gamma)$, one has
\begin{equation*}
\norm{\partial_\nu v}_{L^2(0,T;C(\Gamma))}
\le
C_g\delta_{\mathrm{src}}.
\end{equation*}
In particular,
$$
\norm{\partial_\nu v}_{\YY_\Gamma((t_1,t_2))}
\le C_g\delta_{\mathrm{src}}.
$$
By \Cref{potiin}, we have
\begin{equation*}
\norm{e^{-t_*A}f}_2
\le
C\norm{f}_2^{1-\beta}
\delta_{\mathrm{src}}^\beta.
\end{equation*}
Since $\lambda_1>0$ is the smallest eigenvalue of $A$, one has  $\norm{f}_2\le\lambda_1^{-\sigma}M$. Hence
\begin{equation*}
\norm{e^{-t_*A}f}_2
\le
CM^{1-\beta}\delta_{\mathrm{src}}^\beta.
\end{equation*}
Finally applying \Cref{lemlog} with $h=f$ yields \eqref{stabi}.  Note that  the power $\beta$ inside the intermediate
H\"older estimate changes only the constant in the logarithm and not the final
logarithmic exponent.
\end{proof}

Since $\Dom(A^{1/2})=H_0^1(\Omega)$ and
$\|A^{1/2}f\|_2=\|\nabla f\|_2$, \Cref{insourth} immediately gives the following corollary.

\begin{corollary}[$H_0^1$ source class under trace admissibility]
\label{cor:H1-source}
Assume in addition that
$$
f\in H_0^1(\Omega)\cap\mathfrak D_{\Gamma,T},
\qquad
\|\nabla f\|_{L^2(\Omega)}\le M.
$$
Then
$$
\|f\|_{L^2(\Omega)}
\le
CM
\left[
\log\left(e+\frac{M}{\delta_{\mathrm{src}}}\right)
\right]^{-1/2}.
$$
\end{corollary}

\begin{remark}
A boundary supremum of $\partial_\nu e^{-tA}f$ can be singular as $t\downarrow0$
for rough $f$.  Therefore one cannot simply assert that
$\partial_t\partial_\nu u\in L^2(0,T;C(\Gamma))$ for every
$f\in H_0^1(\Omega)$.  This is why \Cref{insourth} explicitly
separates the spectral prior $f\in\Dom(A^\sigma)$ from the boundary-trace
admissibility condition $f\in\mathfrak D_{\Gamma,T}$.  If a stronger smooth
source class is imposed so that the boundary trace is continuous down to
$t=0$, the admissibility condition is automatic.  Alternatively, one may use
delayed observations and analytic continuation, at the price of a weaker
(double-logarithmic) modulus, as in \cite{ChoulliYamamoto2004}. This also motivates the delayed-time observation  inverse source problem studied in Section \ref{delsour} below.
\end{remark}

\section{Delayed-time Hausdorff boundary measurements for the inverse source problem}
\label{delsour}

We now consider the inverse source problem when boundary measurements are
available only after a strictly positive time.  The argument is inspired by
the time analytic continuation method of Choulli--Yamamoto
\cite{ChoulliYamamoto2004}.  There are, however, two important differences.
First, the observation set may have zero surface measure and the observation
norm is therefore the Hausdorff-compatible norm
$L^2_tC_x(\Gamma)$.  Second, the quantity measured here is the time
derivative of the Dirichlet flux, which is precisely the quantity entering
the Volterra reduction of the separated source problem.

Throughout this section,
$$
A=-\Delta: H_0^1(\Omega) \cap H^2(\Omega) \to L^2(\Omega)
$$
denotes the positive Dirichlet Laplacian on a bounded $C^{1,1}$ domain
$\Omega\subset\mathbb R^d$, $d\ge2$, and
$$
X_\Gamma:=C(\Gamma;\mathbb C).
$$
We assume that $\Gamma\subset\partial\Omega$ satisfies the {\bf Hypothesis}  \ref{hyg} in Section 1.

Fix
$$
0<\tau<T.
$$
We consider \eqref{2source} with $g$ satisfying conditions listed in Theorem \ref{deHaus}.

The proof of Theorem \ref{deHaus} is divided into three steps.
The first gives the precise
near-zero boundary smoothing required by the Hausdorff trace norm.  The
second is a quantitative delayed-to-full analytic continuation lemma.  The
third combines that lemma with the Volterra equation, BHSI and the
backward logarithmic estimate.

\subsubsection*{Step 1. Sectorial boundary smoothing near $t=0$}

We first establish the a priori analytic bound which replaces the uniform boundary bound used in the classical delayed-observation argument.

\begin{lemma}
\label{nortra}
Assume
$
\sigma>\frac d4.
$
Then one can choose
$
0<\beta<\frac12
$
and $\theta_0\in(0,\theta)$ such that, for every
$h\in\operatorname{Dom}(A^\sigma)$, the map
\begin{equation}\label{5.1p}
Q_h(z)
:=
\left.
\partial_\nu e^{-zA}h
\right|_\Gamma
\end{equation}
is $X_\Gamma$-valued holomorphic in
$\mathcal S_{\theta_0,R}$ and satisfies
\begin{equation}
\|Q_h(z)\|_{X_\Gamma}
\le
C|z|^{-\beta}
\|A^\sigma h\|_{L^2(\Omega)},
\qquad
z\in\mathcal S_{\theta_0,R}.
\label{secnortr}
\end{equation}
\end{lemma}

\begin{proof}
Choose $p>d$.  Set
$$
\gamma_p
:=
\frac d2\left(\frac12-\frac1p\right),
\qquad
r_p:=1+\gamma_p.
$$
By  $L^2$ to $L^p$ smoothing estimate for the Dirichlet heat
semigroup $e^{-tA}$, one has for $0<t\le T$ that
\begin{equation*}
\|e^{-tA}\|_{L^2\to L^p}
\le
Ct^{-\gamma_p},
\qquad
\|Ae^{-tA}\|_{L^2\to L^p}
\le
Ct^{-1-\gamma_p}.
\end{equation*}
Since $\Omega$ is $C^{1,1}$, the global $W^{2,p}$ estimate for the
Dirichlet Laplacian $A$ yields
$$
\|w\|_{W^{2,p}(\Omega)}
\le
C\bigl(
\|Aw\|_{L^p(\Omega)}
+
\|w\|_{L^p(\Omega)}
\bigr).
$$
Because $p>d$,
$$
W^{2,p}(\Omega)\hookrightarrow C^{1,\alpha}(\overline\Omega),
\qquad
\alpha=1-\frac dp>0.
$$
Therefore,
\begin{equation}
\left\|
\partial_\nu e^{-tA}v
\right\|_{C(\partial\Omega)}
\le
Ct^{-r_p}\|v\|_{L^2(\Omega)},
\qquad
0<t\le T.
\label{norsmoo}
\end{equation}

We now use the spatial regularity of $h$.  Write
$H=A^\sigma h$ and $h=A^{-\sigma}H$.
Since $A\ge\lambda_1>0$,
$$
A^{-\sigma}
=
\frac1{\Gamma(\sigma)}
\int_0^\infty
s^{\sigma-1}e^{-sA}\,ds
$$
in the operator norm on $L^2(\Omega)$.  Hence we have
$$
\partial_\nu e^{-tA}h
=
\frac1{\Gamma(\sigma)}
\int_0^\infty
s^{\sigma-1}
\partial_\nu e^{-(t+s)A}H\,ds.
$$
Using \eqref{norsmoo} for $t+s\le1$, and the
spectral gap for $t+s\ge1$, gives the estimates:
\begin{equation}
\left\|
\partial_\nu e^{-tA}A^{-\sigma}
\right\|_{L^2\to C(\partial\Omega)}
\le
C
\begin{cases}
t^{\sigma-r_p},&\sigma<r_p,\\
1+|\log t|,&\sigma=r_p,\\
1,&\sigma>r_p,
\end{cases}
\qquad 0<t\le T.
\label{fract}
\end{equation}

We observe that
$$
\lim_{p\downarrow d}r_p
=
1+\frac d4-\frac12
=
\frac{d+2}{4}.
$$
Since $\sigma>d/4$, we may choose $p>d$ sufficiently close to
$d$ so that
$r_p-\sigma<\frac12$.
We then choose
\begin{equation*}
\max\{0,r_p-\sigma\}<\beta<\frac12.
\end{equation*}
The logarithmic endpoint in
\eqref{fract} is also dominated by $Ct^{-\beta}$.
Thus
\begin{equation}
\left\|
\partial_\nu e^{-tA}h
\right\|_{C(\partial\Omega)}
\le
Ct^{-\beta}\|A^\sigma h\|_2.
\label{realsec}
\end{equation}

It remains to pass to complex time.  Fix
$\theta_0\in(0,\theta)$, and set
$$
c_{\theta_0}
:=
\frac12\cos\theta_0>0.
$$
For $z\in\mathcal S_{\theta_0,R}$, set
$s_z=c_{\theta_0}|z|$.
Then there holds
$\operatorname{Re}(z-s_z)
\ge
c_{\theta_0}|z|>0$.
Since $A$ is positive self-adjoint, we arrive at
$$
\|e^{-(z-s_z)A}\|_{L^2\to L^2}\le1.
$$
Using the commutation of $A^{-\sigma}$ with the semigroup, we find
$$
\partial_\nu e^{-zA}A^{-\sigma}
=
\partial_\nu e^{-s_zA}A^{-\sigma}
e^{-(z-s_z)A}.
$$
Therefore \eqref{realsec} gives
$$
\left\|
\partial_\nu e^{-zA}A^{-\sigma}
\right\|_{L^2\to C(\partial\Omega)}
\le
C|z|^{-\beta}.
$$
This proves \eqref{secnortr}.  Holomorphy follows from
the analyticity of $e^{-zA}$ and the locally uniform operator bound just
obtained.
\end{proof}

We next apply this estimate to the actually measured quantity.
The following lemma gives the sectorial analyticity of the measured flux.
\begin{lemma}
\label{flux}
Let the assumptions of
Theorem~\ref{deHaus} hold, and let $\beta$ be as
in Lemma~\ref{nortra}.  Define, for
$z\in\mathcal S_{\theta_0,R}$,
\begin{equation}
P_f(z)
:=
g(0)Q_f(z)
+
z\int_0^1
g'\bigl(z(1-r)\bigr)Q_f(rz)\,dr,
\label{compflux}
\end{equation}
where $Q_f$ is defined in \eqref{5.1p}.
Then $P_f$ is $X_\Gamma$-valued holomorphic and
\begin{equation*}
\|P_f(z)\|_{X_\Gamma}
\le
CM|z|^{-\beta},
\qquad
z\in\mathcal S_{\theta_0,R}.
\end{equation*}
For $t\in(0,T)$,
\begin{equation}
P_f(t)
=
\left.
\partial_t\partial_\nu u(t)
\right|_\Gamma.
\label{reide}
\end{equation}
\end{lemma}

\begin{proof}
By Lemma~\ref{nortra},
$$
\|Q_f(rz)\|_{X_\Gamma}
\le
CM(r|z|)^{-\beta}.
$$
Since $\beta<1$,
$$
\int_0^1r^{-\beta}\,dr<\infty.
$$
Thus the Bochner integral in
\eqref{compflux} is well defined.  On every compact subset
of the sector its integrand is dominated by an $L^1(0,1)$ function
independent of $z$.  Standard differentiation under the integral sign
therefore proves holomorphy.

Furthermore, we have
$$
\begin{aligned}
\|P_f(z)\|_{X_\Gamma}
&\le
|g(0)|\,CM|z|^{-\beta}
+
CGM|z|
\int_0^1(r|z|)^{-\beta}\,dr
\\
&\le
CM\bigl(
|z|^{-\beta}+|z|^{1-\beta}
\bigr)
\le
CM|z|^{-\beta},
\end{aligned}
$$
because $|z|<R$.

For real $t>0$, Duhamel's formula can be written as
$$
u(t)
=
\int_0^t
g(t-s)e^{-sA}f\,ds.
$$
Since $Q_f(s)=\partial_\nu e^{-sA}f|_\Gamma$ is locally integrable near
$0$ by $\beta<1$, differentiation in $t$
gives
$$
\partial_t\partial_\nu u(t)
=
g(0)Q_f(t)
+
\int_0^t
g'(t-s)Q_f(s)\,ds.
$$
The change of variables $s=rt$ gives precisely
\eqref{reide}.
\end{proof}

\subsubsection*{Step 2. Quantitative delayed time to full time analytic continuation}

We now prove the analytic lemma which yields  the additional logarithm.

\begin{lemma}
\label{fullsector}
Let $X$ be a complex Banach space, let
$0<\tau<T<R$, and let
$
0\le\beta<\frac12.
$
Suppose that
$$
F:\mathcal S_{\theta_0,R}\longrightarrow X
$$
is holomorphic and
\begin{equation}
\|F(z)\|_X
\le
K|z|^{-\beta},
\qquad
z\in\mathcal S_{\theta_0,R}.
\label{F}
\end{equation}
Set
\begin{equation}\label{Varp}
\varepsilon_\tau
:=
\|F\|_{L^2(\tau,T;X)}.
\end{equation}
Then there exist constants $C,D,\chi>0$, depending only on $
\theta_0,R,\tau,T,\beta$, such that
\begin{equation}
\|F\|_{L^2(0,T;X)}
\le
CK
\left[
\log\left(
e+\frac{DK}{\varepsilon_\tau}
\right)
\right]^{-\chi}.
\label{delsec}
\end{equation}
If $\varepsilon_\tau=0$, then $F\equiv0$.
\end{lemma}

\begin{proof}
The case $\varepsilon_\tau=0$ follows first.  Since $F$ is continuous
on $(\tau,T)$, it vanishes there.  For every $\ell\in X^*$,
$\ell\circ F$is a scalar holomorphic function which vanishes on a real
interval. Hence $\ell(F(z))=0$ throughout the sector.  By Hahn--Banach theorem,
$F\equiv0$.

We henceforth assume $\varepsilon_\tau>0$.  Choose the principal branch
of $z^\beta$ in the sector and define
$$
G(z):=\frac{z^\beta F(z)}{K}.
$$
After multiplying $K$ by an inessential fixed constant, we may assume
\begin{equation}
\sup_{\mathcal S_{\theta_0,R}}\|G(z)\|_X\le1.
\label{G}
\end{equation}
Moreover,
\begin{equation}
\|G\|_{L^2(\tau,T;X)}
\le
T^\beta\frac{\varepsilon_\tau}{K}
=:
\varepsilon.
\label{Gsmall}
\end{equation}

If $\varepsilon$ is bounded from below by a fixed positive constant,
then \eqref{delsec} follows directly from
\eqref{F}, since for $\beta<1/2$,
$$
\int_0^Tt^{-2\beta}\,dt<\infty.
$$
We may therefore assume
$$
0<\varepsilon<\varepsilon_0
$$
with $\varepsilon_0>0$ fixed and sufficiently small.

Let
$$
t_0:=\frac{\tau+T}{2}.
$$
Since the compact interval $[\tau,T]$ lies a positive distance from the
boundary of the sector $\mathcal{S}_{\theta_0,R}$,
using Lemma~\ref{anpro}, applied after scalarization by
elements of $X^*$, we have there exist a disk
$$
D_0:=D(t_0,\rho_0)
\Subset\mathcal S_{\theta_0,R}
$$
and constants $C_0>0$, $\vartheta_0\in(0,1)$ such that
\begin{equation}
\sup_{z\in D_0}\|G(z)\|_X
\le
C_0\varepsilon^{\vartheta_0}.
\label{first}
\end{equation}
Indeed, the scalar estimate is uniform over all
$\ell\in X^*$ with $\|\ell\|_{X^*}\le1$, and Hahn--Banach then yields \eqref{first}.

We next propagate \eqref{first} toward the vertex
$z=0$.  Choose $\kappa>0$ sufficiently small, depending only on
$\theta_0$, and choose $\rho\in(0,1)$ so close to $1$ that
$$
1-\rho<\frac{\kappa}{4}.
$$
After decreasing $\rho_0$, if necessary, we may assume that
$$
D(t_0,\kappa t_0)\subset D_0.
$$
Set
$$
t_n:=\rho^nt_0,
\qquad
D_n:=D(t_n,\kappa t_n),
\qquad n\ge0.
$$
All the disks $D_n$ are contained in
$\mathcal S_{\theta_0,R}$.

There exist constants
$$
C_1>0,\qquad \vartheta_1\in(0,1),
$$
independent of $n$, such that
\begin{equation}
\sup_{D_{n+1}}\|G\|_X
\le
C_1
\left(
\sup_{D_n}\|G\|_X
\right)^{\vartheta_1}.
\label{chain}
\end{equation}
To see this, rescale $z=t_n\zeta$.  The pair of disks
$$
D(1,\kappa),
\qquad
D(\rho,\kappa\rho)
$$
and a fixed complex neighborhood containing them are independent of
$n$.  Applying Lemma~\ref{anpro} on
this fixed geometry to
$\ell(G(t_n\cdot))$, using
\eqref{G} as $M$, and then taking the supremum
over $\ell$ in the unit ball of $X^*$, we obtain
\eqref{chain}.

Iterating \eqref{chain} and using
\eqref{first}, we obtain
\begin{equation}
\sup_{D_n}\|G\|_X
\le
C_2
\varepsilon^{c_0\vartheta_1^n}
\label{itedisma}
\end{equation}
with constants $C_2,c_0>0$ independent of $n$.

Define
$$
a:=\frac{\log\vartheta_1}{\log\rho}>0.
$$
Since
$$
\vartheta_1^n
=
(\rho^n)^a
=
\left(\frac{t_n}{t_0}\right)^a,
$$
and since $1-\rho<\kappa/4$, every
$t\in[t_{n+1},t_n]$ belongs to $D_n$.  Therefore
\eqref{itedisma} implies
\begin{equation}
\|G(t)\|_X
\le
C_2
\varepsilon^{c_1t^a},
\qquad
0<t<\tau,
\label{bacan}
\end{equation}
where $c_1>0$ depends only on the fixed geometry.

Let
\begin{align}\label{5.6g}
L:=\log\frac1\varepsilon.
\end{align}
For $\varepsilon$ sufficiently small, $L\ge1$,
\eqref{bacan} yields
$$
\|F(t)\|_X
\le
CKt^{-\beta}
e^{-c_1Lt^a},
\qquad 0<t<\tau.
$$
Therefore,
\begin{align*}
\|F\|_{L^2(0,\tau;X)}^2
&\le
CK^2
\int_0^\tau
t^{-2\beta}e^{-2c_1Lt^a}\,dt
\nonumber\\
&\le
CK^2
L^{-(1-2\beta)/a}.
\end{align*}
Indeed, the change of variables
$
s=2c_1Lt^a
$
gives
$$
\int_0^\tau
t^{-2\beta}e^{-2c_1Lt^a}\,dt
\le
C
L^{-(1-2\beta)/a}
\int_0^\infty
s^{\frac{1-2\beta}{a}-1}e^{-s}\,ds,
$$
and the last integral is finite precisely because
$\beta<1/2$.

Thus, with
$$
\chi:=\frac{1-2\beta}{2a}>0,
$$
we have
\begin{equation}
\|F\|_{L^2(0,\tau;X)}
\le
CKL^{-\chi}.
\label{2Log}
\end{equation}
Recall $\varepsilon_\tau$ defined  in \eqref{Varp}, $L$ in \eqref{5.6g}), $\varepsilon$ in \eqref{Gsmall}.
Since $e^{-L}\le C_\chi L^{-\chi}$ for $L\ge1$,   by \eqref{2Log}, one obtains
\eqref{delsec}.
\end{proof}

\subsubsection*{Step 3. Volterra inversion and the second logarithm}

We now complete the proof of
Theorem~\ref{deHaus}.

\begin{proof}[Proof of Theorem~\ref{deHaus}]
Let
$$
P(t)
:=
\left.
\partial_t\partial_\nu u(t)
\right|_\Gamma.
$$
By Lemma~\ref{flux}, $P$ is the restriction to the
positive real axis of an $X_\Gamma$-valued holomorphic function satisfying
$$
\|P(z)\|_{X_\Gamma}
\le
CM|z|^{-\beta}
$$
for some $\beta<1/2$.  Applying
Lemma~\ref{fullsector} with
$$
K=CM,
\qquad
\varepsilon_\tau=\delta_\tau,
$$
we obtain constants $C,D,\chi>0$ such that
\begin{equation}
\|P\|_{L^2(0,T;X_\Gamma)}
\le
CM
\left[
\log\left(
e+\frac{DM}{\delta_\tau}
\right)
\right]^{-\chi}.
\label{fullP}
\end{equation}
Set
$$
v(t):=e^{-tA}f,
\qquad
Q(t):=
\left.
\partial_\nu v(t)
\right|_\Gamma.
$$
Lemma~\ref{nortra} and $\beta<1/2$ imply
$$
Q\in L^2(0,T;X_\Gamma).
$$
For real $t\in(0,T)$, the identity
\eqref{reide} is
\begin{equation}
P(t)
=
g(0)Q(t)
+
\int_0^t
g'(t-s)Q(s)\,ds.
\label{delavo}
\end{equation}
Lemma ~\ref{lem:volterra} therefore gives
\begin{equation}
\|Q\|_{L^2(0,T;X_\Gamma)}
\le
C_g
\|P\|_{L^2(0,T;X_\Gamma)}.
\label{Qcontrol}
\end{equation}
Combining
\eqref{fullP} and
\eqref{Qcontrol}, one has
\begin{equation}
\|Q\|_{L^2(0,T;X_\Gamma)}
\le
CM
\left[
\log\left(
e+\frac{DM}{\delta_\tau}
\right)
\right]^{-\chi}.
\label{Qlog}
\end{equation}

Choose once and for all
$0<t_1<t_*<t_2<T.$ Proposition~\ref{potiin}  gives
\begin{equation*}
\|e^{-t_*A}f\|_{L^2(\Omega)}
\le
C
\|f\|_{L^2(\Omega)}^{1-\vartheta}
\|Q\|_{L^2(t_1,t_2;X_\Gamma)}^\vartheta
\end{equation*}
for some fixed $\vartheta\in(0,1)$.

Since $A\ge\lambda_1>0$,
$$
\|f\|_{L^2(\Omega)}
\le
\lambda_1^{-\sigma}
\|A^\sigma f\|_{L^2(\Omega)}
\le
\lambda_1^{-\sigma}M.
$$
Hence  from \eqref{Qlog}, there holds
\begin{equation}
\|e^{-t_*A}f\|_{L^2(\Omega)}
\le
CM
\left[
\log\left(
e+\frac{DM}{\delta_\tau}
\right)
\right]^{-\kappa},
\qquad
\kappa:=\chi\vartheta>0.
\label{log}
\end{equation}
Finally, by
Lemma~\ref{lemlog}, one obtains
\begin{equation}
\|f\|_{L^2(\Omega)}
\le
CM
\left[
\log\left(
e+
\frac{M}{
\|e^{-t_*A}f\|_{L^2(\Omega)}
}
\right)
\right]^{-\sigma}.
\label{debackla}
\end{equation}
Let
$$
L_\tau
:=
\log\left(
e+\frac{DM}{\delta_\tau}
\right).
$$
If $L_\tau$ is sufficiently large,
\eqref{log} gives
$$
\frac{M}{\|e^{-t_*A}f\|_2}
\ge
cL_\tau^\kappa,
$$
and hence
$$
\log\left(
e+
\frac{M}{\|e^{-t_*A}f\|_2}
\right)
\ge
c_1
\log(e+L_\tau).
$$
Substitution into \eqref{debackla} gives
$$
\|f\|_2
\le
CM
\left[
\log\left(
e+
\log\left(
e+\frac{DM}{\delta_\tau}
\right)
\right)
\right]^{-\sigma}.
$$
If $L_\tau$ remains bounded, the same conclusion follows after increasing
$C$, using the a priori estimate
$$
\|f\|_2\le\lambda_1^{-\sigma}M.
$$
This proves \eqref{dedoub}.

If $\delta_\tau=0$, then
$P(t)=0$ on $(\tau,T)$.  By analyticity,
$P\equiv0$ in the sector.  Equation
\eqref{delavo} and uniqueness for the Volterra equation imply
$Q\equiv0$ on $(0,T)$.  Proposition
\ref{potiin} then gives
$e^{-t_*A}f=0$, and the injectivity of the heat semigroup yields
$f=0$.
\end{proof}

\section{Proofs of the main interior observation results}\label{proof}
\subsection{Proof of Theorem~\ref{thmgene} }
The key idea is to employ the Duhamel principle to link the inverse source problem with the observability inequality established in Green et al. \cite{green2024observability}. Indeed, for sources of separable form $g(t)f(x)$, the solution $u(x,t)$ of equation~\eqref{mainpde} can be expressed as
\begin{equation}\label{eq-duha}
u(x,t) = \int_0^t g(t-s)\, v(x,s)ds,
\end{equation}
where $v$ satisfies the homogeneous heat equation
$$
\begin{cases}
\partial_t v - \Delta v = 0, & (x,t)\in\Omega_T,\\
v(x,0)=f(x), & x\in\Omega,\\
v(x,t)=0, & (x,t)\in\partial\Omega\times(0,T).
\end{cases}
$$
Now we are ready to give the proof of Theorem \ref{thmgene} .
\begin{proof}[\bf Proof of Theorem \ref{thmgene} ]
Differentiating \eqref{eq-duha} with respect to $t$ gives
$$
\partial_t u(x,t) = g(0)v(x,t) + \int_0^t g'(t-s) v(x,s)\,ds.
$$
From this one derives a pointwise estimate
$$
|v(x,t)| \le C\bigl|\partial_t u(x,t)\bigr| + C\int_0^t \bigl|v(x,s)\bigr| e^{t-s}\,ds.
$$
by the Gronwall inequality, we further obtain 
\begin{equation}\label{qgronwall}
\int_0^T\sup_{x\in\omega}|v(x,t)|\,dt
\leq
C_g
\int_0^T
\sup_{x\in\omega}|\partial_tu(x,t)|\,d t.
\end{equation}

Applying the observability inequality Theorem 1.3 from~\cite{green2024observability}, that is,
$$
\|v(\cdot,T)\|_{L^2(\Omega)} \le C e^{C/T} \int_0^T \sup_{x\in\omega}|v(x,t)|\,dt,
$$
and combining with the previous bound \eqref{qgronwall}, we obtain
\begin{equation}\label{qvT-final}
\|v(\cdot,T)\|_{L^2(\Omega)} \le C_T \int_0^T \sup_{x\in\omega} \bigl|\partial_t u(x,t)\bigr|\,dt.
\end{equation}

Next, it remains to estimate $\|f\|_{L^2(\Omega)}$ in terms of
$\|v(\cdot,T)\|_{L^2(\Omega)}$. Without loss of generality, we assume that $\| v(T) \|_{L^2(\Omega)} < \| f \|_{H_0^1(\Omega)} \le M$; otherwise, the Poincaré inequality directly yields a Lipschitz-type
stability estimate. Applying the standard conditional
logarithmic stability estimate for the backward heat equation
\cite[Eq.~(7.36), p.~134, with $\alpha=2$ and $s=1$]{bal2012introduction},
we obtain
$$\|f\|_{L^2(\Omega)}
\leq C M
\left[
\log\left(
\frac{M}{\|v(\cdot,T)\|_{L^2(\Omega)}}
\right)
\right]^{-\frac12},
$$
Combining this inequality with estimate~\eqref{qvT-final} completes
the proof of Theorem~\ref{thmgene}.

\end{proof}

\subsection{Proof of Theorem~\ref{ganal}}
Before giving the proof of the main theorem for analytic sources, we need a bound on the high-frequency truncation of the observation data.
\begin{lemma}\label{lem:hf}
Let $u$ solve the heat equation \eqref{mainpde} with source $g(t)f(x)$, where $g\in L^1(0,T)$ and $f\in G^{1,\rho}(\Omega)$. Then for any $\rho_1\in(0,\rho)$ and integer $N\ge1$, there exists a constant $C>0$ depending only on  $d,\rho-\rho_1$ such that
\begin{align}\label{Hubei}
\|(I-\Pi_{\lambda_N})u(\cdot,T)\|_{L^\infty(\Omega)}
\le C\, e^{-\rho_1\sqrt{\lambda_{N+1}}} \|f\|_{G^{1,\rho}}\|g\|_{L^1(0,T)},
\end{align}
where $\Pi_{\lambda_N}$ is the projection onto the first $N$ eigenmodes.
\end{lemma}
\begin{proof}
The proof follows from standard Sobolev embedding and the decay of Gevrey coefficients. Indeed, using the representation
$$
u(x,T) =\int^T_0 e^{(T-s)\Delta}(g(s)f(x))ds,
$$
we obtain by the Sobolev embedding $H^{k_d}(\Omega)\subset\subset L^\infty(\Omega)$ with $k_d:=[\frac{d}{2}] + 1$ and the fact $\|e^{t\Delta}\|_{L^{\infty}(\Omega)\to L^{\infty}(\Omega)}\le C_d$ that
\begin{align*}
\|(I-\Pi_{\lambda_N})u(\cdot,T)\|_{L^\infty(\Omega)}
\le& C_d\int_0^T \|(I-\Pi_{\lambda_N})(g(s)f)\|_{L^\infty(\Omega)}\,ds
\\
\le& C\int_0^T \|(I-\Pi_{\lambda_N})(g(s)f)\|_{H^{k_d}(\Omega)}\,ds
\\
\le& C\|g\|_{L^1(0,T)}\|(I-\Pi_{\lambda_N})f\|_{H^{k_d}(\Omega)}.
\end{align*}
Recall that
$$\|h\|_{H^{k}}\lesssim \|(I-\Delta)^{\frac{k}{2}}h\|_{L^2}, \mbox{ }\forall \ h\in D\left((-\Delta)^{[\frac{k}{2}]}\right),$$
where
$$
D\left((-\Delta)^{[\frac{k}{2}]}\right):=
\left\{h\in H^k(\Omega): \Delta^j f|_{\partial \Omega}=0, \, j=0,...,[\frac{k}{2}]\right\}.
$$
Note that $f\in G^{1,\rho}(\Omega)$ implies
$$
(-\Delta)^{m}f=\sum^{\infty}_{j=0}f_j\lambda^m_j \phi_j(x),
$$
where $\phi_j\in D((-\Delta)^{\ell})$ for all $\ell\in \mathbb N$, since $\phi_j$ is an eigenfunction to $-\Delta$ with Dirichlet boundary condition.
Thus $f\in D((-\Delta)^\ell)$ for any $\ell\in \mathbb N$. And by the Parseval identity, one gets for any $\rho_1\in(0,\rho)$ that there exists a constant $C>0$ such that
\begin{align*}
\|(I-\Pi_{\lambda_N})f\|^2_{H^{k_d}(\Omega)}\le C\sum^{\infty}_{j> N} (1+\lambda^{k_d}_j)|f_j|^2\le C  e^{-2\rho_1\sqrt{\lambda_{N+1}}}\sum^{\infty}_{j> N}e^{2\rho\sqrt{\lambda_j}}  |f_j|^2,
\end{align*}
which finishes the proof of  Lemma \ref{lem:hf}.
\end{proof}

Now we are ready to give the proof of Theorem \ref{ganal}.
\begin{proof}[\bf Proof of Theorem \ref{ganal}]
Let
$$
b_k(T)
:=
\int_0^T g(s)e^{-\lambda_k(T-s)}\,ds.
$$
Then we have the identity
$$
u(\cdot,T)
=
\sum_{k=1}^\infty b_k(T)f_k\phi_k.
$$
By assumption \eqref{3cog},
\begin{equation}\label{qbk-lowern}
|b_k(T)|
\ge
\frac{c_g}{\lambda_k},
\qquad k\ge1.
\end{equation}
Since the projection $\Pi_\Lambda$ is constant as $\Lambda$ varies between two
consecutive eigenvalues, \eqref{Hubei}  implies that
for every $\rho_1\in(0,\rho)$,
\begin{equation}\label{qu-tailn}
\|(I-\Pi_\Lambda)u(\cdot,T)\|_{L^\infty(\Omega)}
\le
C_{\rho_1}
e^{-\rho_1\sqrt{\Lambda}}
\|f\|_{G^{1,\rho}(\Omega)}
\|g\|_{L^1(0,T)}
\end{equation}
for all $\Lambda\ge1$, after increasing the constant if necessary.

Because $\rho>C_S$, fix once and for all
$C_S<\rho_1<\rho$.
Set
\begin{equation*}
\delta
:=
\sup_{x\in\omega}|u(x,T)|.
\end{equation*}
We first derive a quantitative estimate valid for every sufficiently
large spectral threshold $\Lambda$.
From \eqref{qbk-lowern},
\begin{align*}
\|\Pi_\Lambda f\|_{L^2(\Omega)}^2
=
\sum_{\lambda_k\le\Lambda}|f_k|^2
 \le
\frac{1}{c_g^2}
\sum_{\lambda_k\le\Lambda}
\lambda_k^2
|b_k(T)f_k|^2\le
\frac{\Lambda^2}{c_g^2}
\|\Pi_\Lambda u(\cdot,T)\|_{L^2(\Omega)}^2.
\end{align*}
Hence, by \Cref{specpro}, one finds
\begin{align*}
\|\Pi_\Lambda f\|_{L^2(\Omega)}
&\le
C\Lambda e^{C_S\sqrt{\Lambda}}
\sup_{x\in\omega}
|\Pi_\Lambda u(x,T)|.
\end{align*}
Moreover, \eqref{qu-tailn} shows
\begin{align*}
\sup_{x\in\omega}
|\Pi_\Lambda u(x,T)|
\le
\sup_{x\in\omega}|u(x,T)|
+
\|(I-\Pi_\Lambda)u(\cdot,T)\|_{L^\infty(\Omega)}
\nonumber\le
\delta
+
C_{\rho_1}
e^{-\rho_1\sqrt{\Lambda}}
\|f\|_{G^{1,\rho}(\Omega)}
\|g\|_{L^1(0,T)}.
\end{align*}
Using the a priori bound
$$
\|f\|_{G^{1,\rho}(\Omega)}\le M,
$$
and absorbing the fixed norm of $g$ into the constant, we conclude that
\begin{equation}\label{qlow-f-finaln}
\|\Pi_\Lambda f\|_{L^2(\Omega)}
\le
C\Lambda e^{C_S\sqrt{\Lambda}}\delta
+
CM\Lambda
e^{-(\rho_1-C_S)\sqrt{\Lambda}}.
\end{equation}
Note that
the spectral Gevrey norm prior directly controls the complementary spectral tail. In fact,
\begin{align}
\|(I-\Pi_\Lambda)f\|_{L^2(\Omega)}^2
=
\sum_{\lambda_k>\Lambda}|f_k|^2
\le
e^{-2\rho\sqrt{\Lambda}}
\sum_{\lambda_k>\Lambda}
e^{2\rho\sqrt{\lambda_k}}|f_k|^2
\nonumber\le
M^2e^{-2\rho\sqrt{\Lambda}}.
\label{qf-highn}
\end{align}
Therefore \eqref{qlow-f-finaln} and \ref{qf-highn} give
\begin{equation}\label{maetj}
\|f\|_{L^2(\Omega)}
\le
C\Lambda e^{C_S\sqrt{\Lambda}}\delta
+
CM\Lambda
e^{-(\rho_1-C_S)\sqrt{\Lambda}}
+
M e^{-\rho\sqrt{\Lambda}}.
\end{equation}

We now optimize \eqref{maetj}.  If $\delta=0$, then
letting $\Lambda\to\infty$ gives
$
\|f\|_{L^2(\Omega)}=0,
$
since $\rho_1>C_S$.  Hence the asserted estimate holds trivially.

Thus assume that $\delta>0$. By the definition of spectral Gevrey norm,
\begin{equation}\label{qL2-apriorin}
\|f\|_{L^2(\Omega)}\le M.
\end{equation}
Hence, if $\delta\ge M/2$, then for every $\alpha\in(0,1)$,
\begin{equation}\label{qlarge-deltan}
\|f\|_{L^2(\Omega)}
\le
M
\le
2^\alpha M^{1-\alpha}\delta^\alpha.
\end{equation}
It remains to consider
$0<\delta<M/2$. Set
\begin{equation}  \label{choicerr}
r:=\frac{\delta}{M}\in(0,1/2), \qquad
\sqrt{\Lambda}
=
\frac{1}{\rho_1}
\log\frac{1}{r}.
\end{equation}
For $r$ sufficiently small this gives $\Lambda\ge1$. Then the remaining
bounded range of $r$ also gives \eqref{qlarge-deltan} with $2^{\alpha}$ replaced by some larger constant.

Now we set
\begin{equation*}
\alpha_0
:=
1-\frac{C_S}{\rho_1}
=
\frac{\rho_1-C_S}{\rho_1}
>0.
\end{equation*}
By \eqref{choicerr}, we have
\begin{align}
e^{C_S\sqrt{\Lambda}}r=
r^{\alpha_0},
\qquad
e^{-(\rho_1-C_S)\sqrt{\Lambda}}
=
r^{\alpha_0},
\qquad
e^{-\rho\sqrt{\Lambda}}
=
r^{\rho/\rho_1}.
\label{qthird-balancen}
\end{align}
Since $\rho>\rho_1$, one has
$\frac{\rho}{\rho_1}>\alpha_0$.
Substituting  \eqref{qthird-balancen}
into \eqref{maetj} yields
\begin{equation}\label{qloglossn}
\|f\|_{L^2(\Omega)}
\le
CM
\left(
1+\left|\log r\right|^2
\right)
r^{\alpha_0}.
\end{equation}
Let
\begin{equation*}
\alpha
:=
\frac{\alpha_0}{2}
=
\frac{\rho_1-C_S}{2\rho_1}
>0.
\end{equation*}
Since
$$
\sup_{0<r\le1/2}
\left(
1+|\log r|^2
\right)
r^{\alpha_0-\alpha}
<\infty,
$$
 \eqref{qloglossn} thus implies
\begin{equation}\label{qHolder-relativen}
\|f\|_{L^2(\Omega)}
\le
C
M r^\alpha
=
C M^{1-\alpha}\delta^\alpha.
\end{equation}
Combining \eqref{qHolder-relativen} with \eqref{qlarge-deltan}, we obtain
\begin{equation*}
\|f\|_{L^2(\Omega)}
\le
C M^{1-\alpha}
\left(
\sup_{x\in\omega}|u(x,T)|
\right)^\alpha,
\end{equation*}
where the exponent $\alpha\in (0,1)$   depends only on the spectral inequality
constant $C_S$ and on the choice of $\rho_1\in(C_S,\rho)$, while
$C$ depends, in addition, on  $c_g$, $\|g\|_{L^1(0,T)}$.  Therefore,  Theorem \ref{ganal} is proved.
\end{proof}

\section{Further Application: Inverse Initial Data from Hausdorff Boundary Flux}
\label{sinverinitial}
The boundary Hausdorff spectral inequality and the resulting parabolic interpolation estimate are not restricted to inverse source problems. To demonstrate their broader applicability, we consider in this section the recovery of the initial state for the homogeneous heat equation from boundary flux observations on a Hausdorff-dimensional set. The argument follows directly from Proposition \ref{potiin} and the backward logarithmic estimate in Lemma \ref{lemlog}.

We consider
\begin{equation*}
\begin{cases}
\partial_ty+Ay=0,&0<t<T,\\
y(0)=y_0\in L^2(\Omega),
\end{cases}
\end{equation*}
where $A$ denotes the
Dirichlet Laplacian.

The following result gives the conditional stability for the inverse initial-value problem.
\begin{theorem}
\label{datainverse}
Let $0<t_1<t_*<t_2<T$, $\sigma>0$, and $M>0$.  Assume
\begin{equation}
y_0\in\Dom(A^\sigma),
\qquad
\norm{A^\sigma y_0}_2\le M.
\label{inipri}
\end{equation}
Set
\begin{equation*}
\delta_{\mathrm{ini}}
:=
\norm{\partial_\nu y}_{\YY_\Gamma((t_1,t_2))}.
\end{equation*}
Then
\begin{equation}
{}
\norm{y_0}_{L^2(\Omega)}
\le
CM
\left[
\log\left(e+\frac{M}{\delta_{\mathrm{ini}}}\right)
\right]^{-\sigma},
\label{instab}
\end{equation}
where $C$ depends only on the fixed geometry, the quantitative constants in
BHSI, $t_1,t_*,t_2$, and $\sigma$.
\end{theorem}

\begin{proof}
By \Cref{potiin} with $h=y_0$,
\begin{equation*}
\norm{y(t_*)}_2
\le
C\norm{y_0}_2^{1-\beta}\delta_{\mathrm{ini}}^\beta.
\end{equation*}
Since $\lambda_1>0$ and \eqref{inipri} holds,
$$
\norm{y_0}_2\le\lambda_1^{-\sigma}M.
$$
Hence
\begin{equation}
\norm{y(t_*)}_2
\le
C M^{1-\beta}\delta_{\mathrm{ini}}^\beta.
\label{timM}
\end{equation}
Now we can apply \Cref{lemlog} to $h=y_0$.  If
$\delta_{\mathrm{ini}}\ge cM$, the desired estimate is trivial.  Otherwise,
\eqref{timM} implies
$$
\log\left(e+\frac{M}{\norm{y(t_*)}_2}\right)
\ge
c\log\left(e+\frac{M}{\delta_{\mathrm{ini}}}\right)-C,
$$
and \eqref{instab} follows after adjusting the constants.
\end{proof}

As a corollary of \Cref{datainverse}, we have the following result.
\begin{corollary}[$H_0^1$ initial data]
\label{cor:H1-initial}
If $y_0\in H_0^1(\Omega)$ and $\|\nabla y_0\|_2\le M$, then
$$
\|y_0\|_{L^2(\Omega)}
\le
CM
\left[
\log\left(e+\frac{M}{\delta_{\mathrm{ini}}}\right)
\right]^{-1/2}.
$$
\end{corollary}

\begin{remark}
Li--Yamamoto--Zou \cite{LiYamamotoZou2009} prove logarithmic stability for
initial-temperature recovery from flux data on a relatively open boundary
piece.  \Cref{datainverse} replaces that open observation set by a
Hausdorff set which may have zero surface measure: $\dim_H\Gamma>d-2$ on a
flat patch, or $\dim_H\Gamma>d-1-c_{d+1}$ on a general $C^{1,1}$ boundary.
The observation norm must therefore be a trace norm such as
$L_t^2C_x(\Gamma)$ rather than $L^2(d\sigma\,dt)$.
\end{remark}

\section*{Acknowledgement}

The first author is  supported by the National Natural Science Foundation of China under grants 12422110 and 12371244 and by Ningbo Youth Leading Talent Project under grant 2024QL046. The second author thanks the National Natural Science Foundation of China (12271277), and Ningbo Youth Leading Talent Project (2024QL045). The third author is supported by National Natural Science Foundation of China (Project 12422117), Hong Kong Research Grants Council (15302323) and an internal grant of Hong Kong Polytechnic University (Project ID: P0053938, Work Programme: 4-ZZVA)

\bibliographystyle{abbrv}
\bibliography{refs}

\end{document}